\documentclass[11pt,a4paper,twoside,fleqn]{scrartcl}

\usepackage[english]{babel}
\usepackage{verbatim}
\usepackage{amsmath}
\usepackage{amsthm}
\usepackage{amsfonts}
\usepackage{amssymb}

\usepackage[utf8]{inputenc}
\usepackage{a4wide}
\usepackage{mathrsfs}

\usepackage{pgf}

\theoremstyle{plain}

\newtheorem{thm}{Theorem}[section]
\newtheorem{prop}[thm]{Proposition}
\newtheorem*{prop*}{Proposition}
\newtheorem{lem}[thm]{Lemma}
\newtheorem*{lem*}{Lemma}
\newtheorem{cor}[thm]{Corollary}
\newtheorem*{cor*}{Corollary}

\theoremstyle{definition}

\newtheorem{defn}[thm]{Definition}
\newtheorem*{defn*}{Definition}
\newtheorem{rem}[thm]{Remark}
\newtheorem*{rem*}{Remark}

\newtheorem*{example*}{Example}

\newcommand{\C}{\mathbb{C}}

\newcommand{\N}{\mathbb{N}}

\newcommand{\cO}{\mathcal{O}}
\newcommand{\cP}{\mathcal{P}}

\newcommand{\abs}[1]{\vert #1 \vert}
\newcommand{\abslr}[1]{\left\vert #1 \right\vert}
\newcommand{\norm}[1]{\Vert #1 \Vert}
\newcommand{\normlr}[1]{\left\Vert #1 \right\Vert}

\DeclareMathOperator{\interior}{int}

\DeclareMathOperator{\re}{Re}

\newcommand{\coloneq}{\mathrel{\mathop:}=}
\newcommand{\eqcolon}{=\mathrel{\mathop:}}

\allowdisplaybreaks[1]

\author{Olivier S\`{e}te\footnote{Institute of Mathematics, TU Berlin, Stra{\ss}e des 17. Juni 136, 10623 Berlin, Germany (\texttt{sete@math.tu-berlin.de})
}}
\title{Some properties of Faber--Walsh Polynomials}

\begin{document}
\maketitle

\begin{abstract}
Walsh introduced a generalisation of Faber polynomials to certain compact sets which need not be connected.  We derive several equivalent representations of these Faber--Walsh polynomials, analogous to representations of Faber polynomials.  Some simple asymptotic properties of the Faber--Walsh polynomials on the complement of the compact set are established.  We further show that suitably normalised Faber--Walsh polynomials are asymptotically optimal polynomials in the sense of \cite{EiermannNiethammer1983}.
\end{abstract}



\section{Introduction}

Faber polynomials, first introduced by Georg Faber in 1903 in \cite{Faber1903}, are polynomials associated with a simply connected compact set $E$ (not a single point).
They allow the expansion of functions $f$ analytic on $E$ into a series $f(z) = \sum_{k=0}^\infty a_k F_k(z)$, where $F_k(z)$ is a monic polynomial of degree $k$, the \emph{$k$-th Faber polynomial}.  The Faber polynomials depend only on $E$, while the coefficients $a_k$ depend on the function $f$ (and $E$).  The series shares many properties with power series, in particular it converges uniformly to $f$ on $E$.

Faber polynomials and Faber series expansions have proven useful on many occasions, see e.g. the introduction in \cite{Curtiss1971} for a large number of classical applications, as well as in the more recent literature, for example for the computation of matrix functions $f(A)$ (cf. \cite{MoretNovati2001}).  However, one drawback in the use of Faber polynomials is that these are limited to simply \emph{connected} compact sets.

In his 1958 article \cite{Walsh1958}, Walsh introduced polynomials that generalise Faber polynomials to compact sets consisting of several components (i.e. whose complement is a $\nu$-times connected region).  These generalised Faber polynomials $b_n(z)$ are called Faber--Walsh polynomials.  They allow in particular the expansion of an analytic function $f$ on $E$ into the series $f(z) = \sum_{k=0}^\infty a_k b_k(z)$, where the $b_k(z)$ are (monic) polynomials of degree $k$ depending on $E$ (but not on $f$) and the coefficients $a_k$ depend on $f$ (and $E$).  The series converges uniformly and maximally to $f$ on $E$.

Unfortunately, while the literature on Faber polynomials is quite extensive, Walsh's article \cite{Walsh1958} seems to have been mostly overlooked.  One notable exception is the book \cite{Suetin1976} of Suetin, which contains a proper subsection on Faber--Walsh polynomials and two references \cite{Kiselev1963,Kiselev1965} on work of Kiselev (in Russian).  
Further, Walsh's article is cited in \cite{Eisenbach1961} and \cite{EisenbachTietz1963}, as well as in Grunsky's book \cite{Grunsky1978} (where the results are however not used).


This article is organised as follows.  Section 2 contains the definition of the Faber--Walsh polynomials from \cite{Walsh1958}, along with the necessary preliminaries and notation.  
In section 3, basic properties of the Faber--Walsh polynomilas are investigated.  In particular, several new representations are derived and some new simple asymtotic properties are established.  Further, a recursion for the computation of the Faber--Walsh polynomials is derived from Walsh's proof of existence.  In section 4 it is shown that (suitably normalised) Faber--Walsh polynomials are asymptotically optimal polynomials.  The last section contains an example.

\section{Faber--Walsh polynomials} \label{sect:fw_polys}

In this section Walsh's generalisation of Faber polynomials is presented.

Let us recall very briefly the setting of Faber polynomials.  More detailed expositions may be found e.g. in the books of Suetin \cite{Suetin1998}, Smirnov and Lebedev \cite[Chapter~2]{SmirnovLebedev}, or in the survey \cite{Curtiss1971} by Curtiss.

Let $E$ be a compact set (not a single point) whose complement $K = \widehat{\C} \backslash E$ is simply connected in $\widehat{\C}$.  Then, by the Riemann mapping theorem, there exists a (unique) conformal bijection $\Phi : K \to \{ w \in \widehat{\C} : \abs{w} > \mu \}$ with $\Phi(\infty) = \infty$ and $\Phi'(\infty) = 1$.
The \emph{$n$-th Faber polynomial $F_n(z)$} can be defined as the polynomial part of the Laurent series of $\Phi(z)^n$ at infinity, i.e. $\Phi(z)^n = F_n(z) + \cO(\tfrac{1}{z})$ in a neighbourhood of infinity.  Equivalently, the Faber polynomials can be defined via their generating function
\begin{equation*}
\frac{\psi'(w)}{\psi(w)-z} = \sum\limits_{n=0}^\infty \frac{F_n(z)}{w^{n+1}}, \quad \abs{w} = \sigma > \mu, \quad z \in \interior \psi( \abs{w} = \sigma ),
\end{equation*}
where $\psi = \Phi^{-1}$.  Both definitions are equivalent.  With the above normalisation of $\Phi$, $F_n(z)$ is a monic polynomial of degree $n$.

In \cite{Walsh1958} Walsh introduced polynomials that generalise Faber polynomials to compact sets with several components (see below for the precise assumptions).  This generalisation relies on the existence of a suitable ``canonical domain'' to which $E$ (more precisely $\widehat{\C} \backslash E$) is conformally equivalent and on the corresponding exterior mapping function.  Both where introduced earlier by Walsh \cite{Walsh1956}.

\begin{defn} \label{defn:lemniscatic_dom}
A \emph{lemniscatic domain} is a domain of the form $\{ w \in \widehat{\C} : \abs{U(w)} > \mu \}$, where $\mu > 0$ is some constant and $U(w) = \prod_{j=1}^\nu (w-a_j)^{m_j}$ for some points $a_1, \ldots, a_\nu \in \C$ and some (real) exponents $m_1, \ldots, m_\nu > 0$ with $\sum_{j=1}^\nu m_j = 1$.
\end{defn}

Note that $U$ is an analytic but in general multiple-valued function.  A lemniscatic domain with $\nu = 1$ simply is the exterior of a circle.

The following lemma from Walsh \cite[Lemma~2]{Walsh1958} is of key importance in the sequel.

\begin{lem} \label{lem:existence_alpha_n}
Given a lemniscatic domain as in definition~\ref{defn:lemniscatic_dom}, there exists a sequence $(\alpha_j)_{j=1}^\infty$, chosen from the foci $a_1, \ldots, a_\nu$, with the following property:  For any closed set $C$ not containing any of the points $a_1, \ldots, a_\nu$, there exist constants $A_1(C), A_2(C) > 0$ such that
\begin{equation}
A_1(C) < \tfrac{ \abs{u_n(w)} }{ \abs{U(w)}^n } < A_2(C), \quad \text{for } n = 0, 1, 2, \ldots \text{ and any } w \in C, \label{eqn:double_bound_un}
\end{equation}
where $u_n(w) \coloneq \prod_{j=1}^n (w - \alpha_j)$.
\end{lem}

The sequence $(\alpha_j)_{j=1}^\infty$ can be chosen constructively from $a_1, \ldots, a_\nu$ (see \cite[Lemma~2]{Walsh1958}).  However, in general, this sequence is \emph{not} unique (see section~\ref{sect:ex_fw_polys} below).  If $\nu = 1$, $\alpha_j = a_1$ for all $j$, so that $u_n(w) = (w-a_1)^n$.

\begin{thm}[{cf. \cite[Theorem~3]{Walsh1956}}] \label{thm:existence_ext_mapping}
Let $E_1, \ldots, E_\nu$ be mutually exterior compact sets (none a single point) of the complex plane such that the complement of $E \coloneq \cup_{j=1}^\nu E_j$ in the extended plane is a $\nu$-times connected region (open and connected set).  Then there exists a lemniscatic domain
\begin{equation}
K_1 = \{ w \in \widehat{\C} : \abs{U(w)} > \mu \} \label{eqn:lemniscatic_domain}
\end{equation}
and a conformal bijection
\begin{equation}
\Phi : \widehat{\C} \backslash E \to \{ w \in \widehat{\C} : \abs{U(w)} > \mu \} \text { with } \Phi(\infty) = \infty \text{ and } \Phi'(\infty) = 1. \label{eqn:def_Phi}
\end{equation}
Here $\mu$ is the logarithmic capacity of $E$.  Further, the inverse conformal bijection satisfies
\begin{equation}
\psi \coloneq \Phi^{-1} : \{ w \in \widehat{\C} : \abs{U(w)} > \mu \} \to \widehat{\C} \backslash E \text{ with } \psi(\infty) = \infty \text{ and } \psi'(\infty) = 1. \label{eqn:def_psi}
\end{equation}
\end{thm}

The existence of such a conformal map was first shown by Walsh in \cite[Theorem~3]{Walsh1956}.  The existence of this new type of canonical domain motivated several subsequent publications giving different proofs, e.g. Grunsky \cite{Grunsky1957,Grunsky1957a} and \cite[Theorem~3.8.3]{Grunsky1978}, Jenkins \cite{Jenkins1958} and Landau \cite{Landau1961}.  For some further references see also \cite[pp.~193-194]{Grunsky1978}.  The lemniscatic domain is unique up to a translation in the $w$-plane.  This fact seems to be mentioned only in the MathSciNet-review of \cite{Walsh1958} by T.\,S.~Motzkin.  It follows from the fact that the canonical domain introduced by Walsh is unique up to a M\"{o}bius transformation.  Now, the normalisations $\psi(\infty) = \infty$ and $\psi'(\infty) = 1$ show that this M\"{o}bius transformation necessarily is a translation.

Note that $G_1(w) = \log \abs{U(w)} - \log(\mu)$ is Green's function with pole at infinity for the lemniscatic domain \eqref{eqn:lemniscatic_domain}.
Then $G = G_1 \circ \Phi$ is Green's function with pole at infinity for $\widehat{\C} \backslash E$.  Let us denote their level curves by
\begin{equation}
\begin{aligned}
\Lambda_\sigma &= \{ w \in \C : G_1(w) = \log(\sigma) \} = \{ w \in \C : \abs{U(w)} = \sigma \mu \},
 \quad & \sigma > 1, \\
\Gamma_\sigma &= \{ z \in \C : G(z) = \log(\sigma) \}, \quad & \sigma > 1.
\end{aligned} 
\end{equation}
Note that $\Gamma_\sigma = \psi(\Lambda_\sigma)$ holds.  Further, denote by $E_\sigma$ the interior of $\Gamma_\sigma$ and by $D_\sigma^\infty$ the exterior of $\Lambda_\sigma$ (including the point at infinity).

\begin{thm}[{cf. \cite[Thorem~3]{Walsh1958}}] \label{thm:existence_fw_poly}
Let $E$, $\psi$ and the lemniscatic domain be as in theorem~\ref{thm:existence_ext_mapping}.  Let $(\alpha_j)_{j=1}^\infty$ be as in lemma~\ref{lem:existence_alpha_n}.  Then, for $z \in \Gamma_\sigma$ and any $w \in D_\sigma^\infty$ holds
\begin{equation}
\frac{ \psi'(w) }{ \psi(w) - z } = \sum\limits_{n=0}^\infty \frac{ b_n(z) }{ u_{n+1}(w) }, \quad \text{where }
b_n(z) = \tfrac{1}{2 \pi i} \int_{\Lambda_\lambda} u_n(\tau) \tfrac{ \psi'(\tau) }{ \psi(\tau) - z } \, d\tau \quad ( \sigma < \lambda < \infty) \label{eqn:defn_bk}
\end{equation}
and $u_n(w) = \prod_{j=1}^n (w-\alpha_j)$.  The $b_n(z)$ are monic polynomials of degree $n$.  The polynomial $b_n(z)$ is called the \emph{$n$-th Faber--Walsh polynomial associated with $E$ and $(\alpha_j)_{j=1}^\infty$}.
\end{thm}

Walsh's proof that the $b_n(z)$ are indeed polynomials is a constructive one.  It is presented at the beginning of section~\ref{sect:recursion} below.  In particular an algorithm for the computation of the $b_n(z)$ can be derived from it.

The Faber--Walsh polynomials are independent of the lemniscatic domain and exterior mapping function, as the next lemma shows.  This fact seems to have been overlooked in the literature so far.

\begin{lem} \label{lem:fw_polys_independent_of_K1_psi}
Let $E$ be as in theorem~\ref{thm:existence_ext_mapping} and let $K_1$, $\widetilde{K}_1$ be lemniscatic domains with corresponding $\psi$, $\widetilde{\psi}$ as in \eqref{eqn:def_psi}.  Then there exists a translation $T$ with $T = \psi^{-1} \circ \widetilde{\psi}$.  If $(\alpha_j)_{j=1}^\infty$ and $(\widetilde{\alpha}_j)_{j=1}^\infty$ are as in lemma~\ref{lem:existence_alpha_n} with $T(\widetilde{\alpha}_j) = \alpha_j$, then the polynomials $b_n(z)$ and $\widetilde{b}_n(z)$ defined by \eqref{eqn:defn_bk} satisfy $b_n(z) = \widetilde{b}_n(z)$ for all $n = 0, 1, 2, \ldots$.
\end{lem}

\begin{proof}
The existence of the translation $T$, $w = T(\widetilde{w}) = \widetilde{w} + \beta$, satisfying $T = \psi^{-1} \circ \widetilde{\psi}$ on $\widetilde{K}_1$ has been established in \cite{Walsh1956}, see also the discussion below theorem~\ref{thm:existence_ext_mapping}.  Clearly, $\widetilde{\psi}(\widetilde{w}) = \psi(\widetilde{w} + \beta)$ and $\widetilde{\psi}'(\widetilde{w}) = \psi'(\widetilde{w} + \beta)$.  Note that $a_j = T(\widetilde{a}_j)$ and, since $\alpha_j = T(\widetilde{\alpha}_j)$,  $\widetilde{u}_n(\widetilde{\tau}) = u_n(\widetilde{\tau}+\beta)$.
Let $1 < \sigma < \lambda$.  Then, for any $z \in \Gamma_\sigma$, formula \eqref{eqn:defn_bk} shows
\begin{equation*}
\widetilde{b}_n(z)
= \tfrac{1}{2 \pi i} \int_{ \widetilde{\Lambda}_\lambda } \widetilde{u}_n(\widetilde{\tau}) \tfrac{ \widetilde{\psi}'(\widetilde{\tau}) }{ \widetilde{\psi}(\widetilde{\tau}) - z } \, d\widetilde{\tau}
= \tfrac{1}{2 \pi i} \int_{ \Lambda_\lambda } u_n(\tau) \tfrac{ \psi'(\tau) }{ \psi(\tau) - z } \, d\tau = b_n(z).
\end{equation*}
Here we used that  $w \in \Lambda_\lambda$ if and only if $\widetilde{w} \in \widetilde{\Lambda}_\lambda$, so that if $\widetilde{\gamma}$ is a parametrisation of $\widetilde{\Lambda}_\lambda$, then $\gamma(t) = \widetilde{\gamma}(t) + \beta$ is a parametrisation of $\Lambda_\lambda$.
\end{proof}

Similar to Faber polynomials, Faber--Walsh polynomials allow the series expansion of functions analytic on compact sets.

\begin{thm}[{cf. \cite[Thorem~3]{Walsh1958}}] \label{thm:fw_series}
In the notation of theorem~\ref{thm:existence_fw_poly}, let $f$ be analytic on $E$.  Then there exists a largest number $\rho > 1$ such that $f$ is analytic and single-valued in $E_\rho$.  Inside $E_\rho$ the function $f$ admits a series expansion
\begin{equation}
f(z) = \sum\limits_{k=0}^\infty a_k b_k(z), \quad a_k = \tfrac{1}{2 \pi i} \int\limits_{\Lambda_\lambda} \tfrac{ f(\psi(\tau)) }{ u_{k+1}(\tau) } \, d\tau, \quad 1 < \lambda < \rho. \label{eqn:fw_series}
\end{equation}
The series converges maximally to $f$ on $E$ (and locally uniformly in $E_\rho$ to $f$), i.e.
\begin{equation}
\limsup\limits_{n \to \infty} \norm{ f - \sum_{k=0}^n a_k b_k }_E^{\frac{1}{n}} = \tfrac{1}{\rho}, \label{eqn:asymptotic_error_f-sn}
\end{equation}
and the coefficients satisfy
\begin{equation*}
\limsup\limits_{k \to \infty} \abs{a_k}^{\frac{1}{k}} = \tfrac{1}{\rho \mu}.
\end{equation*}
\end{thm}

Equation \eqref{eqn:asymptotic_error_f-sn} states that the partial sums $s_n(z)$ converge maximally to $f(z)$ on $E$ (in the sense of Walsh \cite[ch.~4]{Walsh-IntApprox}).

\begin{rem}
For $\nu = 1$, i.e. for simply connected $E$, the Faber--Walsh polynomials are the (classical) Faber polynomials and the series \eqref{eqn:fw_series} is the Faber series of $f$.
For $\psi(w) = w$, i.e. $\widehat{\C} \backslash E$ is a lemniscatic domain, we have $b_n(z) = u_n(z)$ and the series \eqref{eqn:fw_series} is a series of interpolation in the points $(\alpha_j)_{j=1}^\infty$ (cf. \cite[Theorem~1]{Walsh1958}).
For $\nu = 1$ and $\psi(w) = w$, i.e. $E$ is a disk, we obtain $b_n(z) = u_n(z) = (z-a_1)^n$ and \eqref{eqn:fw_series} is the Taylor series of $f$.
\end{rem}

In the next two sections we study properties of the Faber--Walsh polynomials themselves.  The approximation of analytic functions by means of (truncated) series of Faber--Walsh polynomials will be considered in the example in the last section.

\section{Some properties of Faber--Walsh polynomials}

This section contains several properties of Faber--Walsh polynomials.  To our knowledge, most of these are not found in the previous literature.

The section is devided in three parts.  The first contains several equivalent representations of Faber--Walsh polynomials. 
The second part is devoted to simple asymptotic studies of the Faber--Walsh polynomials.  The third contains a recursion formula for the computation of the Faber--Walsh polynomials based on their definition via the generating function, cf.~\eqref{eqn:defn_bk}.

\subsection{Equivalent representations}

The following proposition collects several representations of Faber--Walsh polynomials.  All of these are generalisations of corresponding well-known properties of Faber polynomials.

\begin{prop} \label{prop:fw_representations}
Let the notation be as in theorem~\ref{thm:existence_fw_poly}, then the following hold for $n \geq 0$:
\begin{enumerate}
\item Integral representation:  Let $\sigma > 1$.  Then, for $z$ on and interior to $\Gamma_\sigma$ holds
\begin{equation}
b_n(z) = \tfrac{1}{2 \pi i} \int_{\Lambda_\lambda} u_n(\tau) \tfrac{ \psi'(\tau) }{ \psi(\tau) - z } \, d\tau
= \tfrac{1}{2 \pi i} \int_{\Gamma_\lambda} \tfrac{ u_n(\Phi(\zeta))}{ \zeta-z } \, d\zeta, \label{eqn:integral_repres_bk}
\end{equation}
for any $\sigma < \lambda < \infty$.

\item $b_n(z)$ is given by the polynomial part of the Laurent series of $u_n(\Phi(z))$ at infinity.

\item $b_n(z) = \sum\limits_{j=0}^n \beta_j z^j$ where $\beta_j = \tfrac{1}{2 \pi i} \int_{ \abs{z} = R } \tfrac{ u_n(\Phi(z)) }{ z^{j+1} } \, dz$.

\item $b_n(z)$ is the uniquely determined monic polynomial of degree $n$ such that
\begin{equation}
b_n(\psi(t)) = u_n(t) + \sum_{k=1}^\infty \tfrac{ \alpha_{nk} }{ u_k(t) }, \quad t \in K_1. \label{eqn:Entwicklung_bncircpsi}
\end{equation}
The $\alpha_{nk}$ are called the \emph{Faber--Walsh coefficients associated with $E$ and $(\alpha_j)_{j=1}^\infty$} and are given by
\begin{equation}
\alpha_{nk} = - \tfrac{1}{4 \pi^2} \int_{ \Gamma_\lambda } \int_{\Gamma_{\lambda_1}} u_{k-1}(\tau) u_n(s) \tfrac{\psi'(s)}{\psi(s)-\psi(\tau)} \, ds \, d\tau \quad (1 < \lambda_1 < \lambda < \infty). \label{eqn:FW-coefficients}
\end{equation}
\end{enumerate}
\end{prop}

For Faber polynomials, properties 1., 2. and 4. can be found e.g. in \cite{Curtiss1971}, property 3. in \cite[(2.3) and (2.4)]{Ellacott1983}.

\begin{proof}
\begin{enumerate}
\item See \cite[p.~30]{Walsh1958}.

\item By \eqref{eqn:def_Phi}, $\Phi(\zeta) = \zeta + d_0 + \sum_{k=1}^\infty \frac{d_k}{\zeta^k}$ in a neighbourhood of $\zeta = \infty$ and the series converges locally uniformly.  Thus also the Laurent series of
\begin{equation*}
u_n(\Phi(\zeta)) = \prod_{j=1}^n (\Phi(\zeta)-\alpha_j) = p_n(\zeta) + \sum_{k=1}^\infty \tfrac{\widetilde{d}_k}{\zeta^k}
\end{equation*}
converges locally uniformly in a neighbourhood of infinity.  Here $p_n(\zeta)$ is a monic polynomial of degree $n$.  Now, let $z$ in some $\Gamma_\sigma$.  For sufficiently large $\lambda > \sigma$ the integral representation \eqref{eqn:integral_repres_bk} of $b_n(z)$ shows
\begin{equation*}
b_n(z) = \tfrac{1}{2 \pi i} \int_{\Gamma_\lambda} \tfrac{u_n(\Phi(\zeta))}{\zeta-z} \, d\zeta = \tfrac{1}{2 \pi i} \int_{\Gamma_\lambda} \tfrac{p_n(\zeta)}{\zeta-z} \, d\zeta + \tfrac{1}{2 \pi i} \int_{\Gamma_\lambda} \sum_{k=1}^\infty \tfrac{\widetilde{d}_k}{\zeta^k} \tfrac{1}{\zeta-z} \, d\zeta = p_n(z).
\end{equation*}
The second integral vanishes by virtue of Cauchy's integral formula for domains with infinity as interior point (cf. e.g. \cite[Problem~14.14]{Markushevich1965} or \cite[p.~335]{LawrentjewSchabat1967}).

\item Follows from 2. by Cauchy's integral formula.

\item We start with a general observation.  If $P_n(z)$ is a polynomial of (exact) degree $n$ such that
\begin{equation*}
P_n(\psi(t)) = u_n(t) + \sum_{k=1}^\infty \tfrac{\beta_{nk}}{u_k(t)}, \quad t \in K_1,
\end{equation*}
for some $\beta_{nk} \in \C$, then the coefficients of $P_n(z)$ are uniquely determined by $\alpha_1, \ldots, \alpha_n$ and the coefficients $c_0, c_1, \ldots, c_n$ of $\psi$, $\psi(t) = t + c_0 + \sum_{k=1}^\infty \tfrac{c_k}{t^k}$.  To see this, consider the expansion of $P_n \circ \psi$ around infinity and equate the coefficients of $t^n, t^{n-1}, \ldots, t^0$.

We show that the Faber--Walsh polynomials associated with $E$ and $(\alpha_j)_{j=1}^\infty$ admit such an expansion.
Let $\sigma > 1$ and $t \in \Lambda_\sigma$, $z = \psi(t) \in \Gamma_\sigma$.  Let $\lambda_1, \lambda$ with $1 < \lambda_1 < \sigma < \lambda$.  The function $s \mapsto u_n(s) \tfrac{\psi'(s)}{\psi(s)-\psi(t)}$ is analytic in $\{ w \in \C : \lambda_1 \mu \leq \abs{U(w)} \leq \lambda \mu \}$ except at $s = t \in \Lambda_\sigma$, where it has a simple pole with residue $u_n(t)$.
Indeed,
\begin{equation*}
\lim\limits_{s \to t} (s-t) u_n(s) \tfrac{\psi'(s)}{\psi(s)-\psi(t)} = \lim\limits_{s \to t} u_n(s) \tfrac{\psi'(s)}{ (\psi(s)-\psi(t))/(s-t) } = u_n(t).
\end{equation*}
Hence, by \eqref{eqn:defn_bk},
\begin{align*}
b_n(\psi(t)) &= \tfrac{1}{2 \pi i} \int_{\Gamma_\lambda} u_n(s) \tfrac{\psi'(s)}{\psi(s)-\psi(t)} \, ds
= \tfrac{1}{2 \pi i} \int_{\Gamma_\lambda - \Gamma_{\lambda_1} + \Gamma_{\lambda_1}} u_n(s) \tfrac{\psi'(s)}{\psi(s)-\psi(t)} \, ds \\
&= \tfrac{1}{2 \pi i} \int_{\Gamma_\lambda - \Gamma_{\lambda_1}} u_n(s) \tfrac{\psi'(s)}{\psi(s)-\psi(t)} \, ds + \tfrac{1}{2 \pi i} \int_{\Gamma_{\lambda_1}} u_n(s) \tfrac{\psi'(s)}{\psi(s)-\psi(t)} \, ds \\
&= u_n(t) + \underbrace{ \tfrac{1}{2 \pi i} \int_{\Gamma_{\lambda_1}} u_n(s) \tfrac{\psi'(s)}{\psi(s)-\psi(t)} \, ds }_{ \eqcolon G_n(t) }.
\end{align*}
The function $G_n(t)$ is analytic in $D_{\lambda_1}^\infty$. By \cite[Theorem~2]{Walsh1958}, it can be expanded into a series of interpolation of the form
\begin{equation}
G_n(t) = \sum\limits_{k=0}^\infty \tfrac{\alpha_{nk}}{u_k(t)}, \quad \alpha_{nk} = \tfrac{1}{2 \pi i} \int_{ \Gamma_{\lambda'} } u_{k-1}(\tau) G_n(\tau) \, d\tau \quad (\lambda_1 < \lambda' < \infty). \label{eqn:expansion_G_n}
\end{equation}
From $\psi(\infty) = \infty$ we have $G_n(\infty) = 0$ and thus $\alpha_{n,0} = 0$.  For $k \geq 1$, plug $G_n(t)$ in the representation \eqref{eqn:expansion_G_n} of $\alpha_{nk}$ to obtain \eqref{eqn:FW-coefficients}.  This shows that $b_n \circ \psi$ has the form \eqref{eqn:Entwicklung_bncircpsi}.  Note that $1 < \lambda_1 < \sigma$ is arbitrary.  By \cite[Theorem~2]{Walsh1958}, the series in \eqref{eqn:Entwicklung_bncircpsi} converges in $D_1^\infty$ to $b_n \circ \psi$, uniformly on any $D_{\lambda_1}^\infty$, $\lambda_1 > 1$.
\end{enumerate}
\end{proof}

The Faber--Walsh coefficients generalise the \emph{Faber coefficients} of a compact set $E$ whose complement is simply connected (cf. e.g. \cite{Curtiss1971}).  Note that the Faber--Walsh coefficients depend on the choice of the sequence $(\alpha_j)_{j=1}^\infty$ (via the polynomials $u_n(z)$) but are independent of the choice of the lemniscatic domain and the exterior mapping function.  This can be seen as in the proof of lemma~\ref{lem:fw_polys_independent_of_K1_psi}.

\subsection{Some simple asymptotics}

We now turn to some simple asymptotic properties of the Faber--Walsh polynomials.
Let us begin by deriving a representation (formula~\eqref{eqn:asymptotics_of_bn} below) which turns out to be quite useful in the sequel (cf.~propositions~\ref{prop:bn_two-sided_inequality} and \ref{prop:fw_poly_quotient}).  The proof follows a similar proof for Faber polynomials from \cite[pp.~134-135]{SmirnovLebedev}.  Some of these ideas already appear in \cite{Walsh1958} and we adopt Walsh's notation.

Let the notation be as in theorems~\ref{thm:existence_ext_mapping} and \ref{thm:existence_fw_poly}.  Consider the generating function of the Faber--Walsh polynomials.  Note that for $z \in K = \widehat{\C} \backslash E$ there is exactly one $\tau \in K_1$ with $\Phi(z) = \tau$, $z = \psi(\tau)$.
The function $\tfrac{ \psi'(t) }{ \psi(t) - \psi(\tau)}$ of $t, \tau \in K_1$ is analytic in both variables, except for $t = \tau$.  The function has a simple pole at these points. Viewed as a function of $t$, the residue at $t = \tau$ is $1$.
Hence
\begin{equation}
\tfrac{ \psi'(t) }{ \psi(t)-\psi(\tau) } = \tfrac{1}{t-\tau} + \Psi(t,\tau), \label{eqn:defn_Psi}
\end{equation}
where $\Psi(t,\tau)$ is a function analytic for $t$ and $\tau$ in $K_1$.

By \cite[Theorem~2]{Walsh1958}, the function $t \mapsto \Psi(t,\tau)$, analytic in $K_1$, has an expansion
\begin{equation}
\Psi(t,\tau) = \sum_{k=0}^\infty \tfrac{ e_{k-1}(\tau) }{ u_k(t) }, \quad e_{k-1}(\tau) = \tfrac{1}{2 \pi i} \int_{ \abs{U(t)} = \lambda \mu } u_{k-1}(t) \Psi(t,\tau) \, dt, \quad 1 < \lambda < \infty. \label{eqn:expansion_Psi}
\end{equation}
Here $u_{-1}(t) = (t-\alpha_1)^{-1}$.  The series converges uniformly (for fixed $\tau$) on any $D_\sigma^\infty = \{ w : \abs{U(w)} > \sigma \mu \}$, $\sigma > 1$.  Note that the $e_{k-1}(\tau)$ are analytic in $K_1$ (see e.g. \cite[Cor.~5.6.2]{Wegert2012}).
For any fixed $\tau \in K_1$, $\tau \neq \infty$, and $t = \infty$, equations \eqref{eqn:defn_Psi} and \eqref{eqn:expansion_Psi} yield $0 = e_{-1}(\tau)$, so that $e_{-1}$ vanishes identically.
Now, $\tau = \infty$ yields $0 = \sum_{k=1}^\infty \tfrac{ e_{k-1}(\infty) }{ u_k(t) }$.  Since the series is uniformly convergent, all coefficients vanish.

Let $1 < \sigma < \lambda < \infty$ and $\tau \in \Lambda_\sigma$ and $t \in \Lambda_\lambda$. Then, $\tfrac{1}{t-\tau} = \sum_{k=0}^\infty \tfrac{u_k(\tau)}{u_{k+1}(t)}$, as the partial sums of this series are the Hermite-interpolation polynomials to $\tfrac{1}{t-\tau}$ in the points $(\alpha_j)_{j=1}^\infty$, see \cite{Walsh1958}. Thus, \eqref{eqn:defn_Psi} becomes
\begin{equation*}
\sum\limits_{k=0}^\infty \tfrac{ b_k(z) }{ u_{k+1}(t) } = \tfrac{\psi'(t)}{\psi(t)-\psi(\tau)} = \tfrac{1}{t-\tau} + \Psi(t,\tau) = \sum\limits_{k=0}^\infty \tfrac{u_k(\tau)}{ u_{k+1}(t) } + \sum\limits_{k=0}^\infty \tfrac{e_k(\tau)}{u_{k+1}(t)}.
\end{equation*}
For fixed $\tau \in \Lambda_\sigma$, all three series converge uniformly in $t$ in $\abs{U(t)} > \sigma \mu$. From this we conclude $b_k(z) = u_k(\tau) + e_k(\tau)$ for any $\tau \in K_1$, $\tau \neq \infty$, by the uniqueness of the coefficients of such series. The equality extends to the point at infinity.

We estimate the coefficient functions $e_n$.  
Note that for $\tau \in \Lambda_\lambda$, i.e. with $\abs{U(\tau)} = \lambda \mu$, we have $\abs{u_n(\tau)} \leq A_2(\Lambda_\lambda) \abs{U(w)}^n = A_2(\Lambda_\lambda) (\lambda \mu)^n$ by virtue of lemma~\ref{lem:existence_alpha_n}.  Then \eqref{eqn:expansion_Psi} implies
\begin{equation*}
\abs{ e_n(\tau) } \leq \tfrac{ L(\Lambda_\lambda) }{ 2 \pi } \max\limits_{ t \in \Lambda_\lambda } \abs{u_n(t)} \max\limits_{t \in \Lambda_\lambda} \abs{\Psi(t,\tau)}
\leq \underbrace{ \tfrac{ L(\Lambda_\lambda) }{ 2 \pi } A_2(\Lambda_\lambda) \max\limits_{(t,\tau) \in \Lambda_\lambda \times \Lambda_\lambda} \abs{\Psi(t,\tau)} }_{ \eqcolon G(\lambda) } (\lambda \mu)^n,
\end{equation*}
where $L(\Lambda_\lambda)$ denotes the length of the curve $\Lambda_\lambda$.  Note that $G(\lambda) < \infty$, since $\Psi(t, \tau)$ is regular for $t, \tau \in K_1$.
An estimate of $G(\lambda)$ for simply connected $E$ is derived in \cite[p.~135]{SmirnovLebedev}.  Since $e_n$ is regular in $\abs{U(\tau)} \geq \lambda \mu$, the maximum modulus principle implies
\begin{equation*}
\abs{ e_n(\tau) } \leq G(\lambda) (\lambda \mu)^n, \quad \text{for all } \tau \text{ with } \abs{U(\tau)} \geq \lambda \mu.
\end{equation*}
Then, for any $z \in K$ and any $\lambda > 1$ holds
\begin{equation}
b_n(z) = u_n(\tau) + e_n(\tau)
= u_n(\tau) \Big( 1 + \theta_n(z) G(\lambda) \tfrac{ (\lambda \mu)^n }{ \abs{u_n(\tau)} } \Big), \label{eqn:asymptotics_of_bn}
\end{equation}
where
\begin{equation}
\theta_n(z) \coloneq \tfrac{ \abs{u_n(\tau)} }{ u_n(\tau) } e_n(\tau) \tfrac{1}{ G(\lambda) (\lambda \mu)^n }. \label{eqn:def_theta_n}
\end{equation}
Note that $\abs{\theta_n(z)} \leq 1$ for $\abs{U(\tau)} \geq \lambda \mu$, $\tau = \Phi(z)$.

We apply equation~\eqref{eqn:asymptotics_of_bn} to study asymptotics of Faber--Walsh polynomials.

\begin{prop} \label{prop:bn_two-sided_inequality}
Let the notation be as in theorem~\ref{thm:existence_fw_poly}.  For $\sigma > 1$ there exist constants $C_1, C_2 > 0$ such that for any $t$ on or exterior to $\Lambda_\sigma$ the following inequalities hold:
\begin{equation*}
\begin{aligned}
\abs{b_n(\psi(t))} &< C_2 \, \abs{ u_n(t) } && \text{ for all } n, \\
\abs{b_n(\psi(t))} &> C_1 \, \abs{u_n(t)}  && \text{ for all sufficiently large } n.
\end{aligned}
\end{equation*}
\end{prop}

\begin{proof}
Let $\sigma > 1$ and $1 < \lambda < \sigma$.  For $t$ on or exterior to $\Lambda_\sigma$ we have $\abs{U(t)} \geq \sigma \mu > \lambda \mu$, so that $\abs{\theta_n(z)} \leq 1$ for $\theta_n$ from \eqref{eqn:def_theta_n} and $z = \psi(t)$.  By lemma~\ref{lem:existence_alpha_n} there exists $A_1 > 0$ such that $A_1 < \tfrac{ \abs{u_n(\tau)} }{ \abs{U(\tau)}^n }$ for all $\tau$ with $\abs{U(\tau)} \geq \sigma \mu$ and all $n$.  This estimate implies $\tfrac{1}{ \abs{u_n(\tau)} } < \tfrac{1}{A_1} \tfrac{1}{ \abs{U(\tau)}^n } \leq \tfrac{1}{A_1} \tfrac{1}{ (\sigma \mu)^n }$.  Hence
\begin{equation*}
\abs{ \theta_n(z) G(\lambda) \tfrac{ (\lambda \mu)^n }{ \abs{u_n(\tau)} } } \leq \tfrac{ G(\lambda) }{A_1} \big( \tfrac{\lambda \mu}{\sigma \mu} \big)^n \to 0.
\end{equation*}
Now the proposition follows from \eqref{eqn:asymptotics_of_bn}.
\end{proof}

In \cite[p.~29]{Walsh1958} a bound $\abs{b_n(z)} < A_1 (\sigma \mu)^n$ for $z \in \Gamma_\sigma$, $\sigma > 1$, is shown ($A_1 > 0$ being some constant).  In \cite[p.~253]{Suetin1998} a two-sided inequality as in proposition~\ref{prop:bn_two-sided_inequality} appears, but it is not noted that the lower bound holds only for sufficiently large $n$.  Indeed, the lower bound is always positive outside $E$, but some $b_n(z)$ can have a zero outside $E$.  Examples are known for $\nu = 1$, i.e. for Faber polynomials, see \cite{Goodman1975} and \cite{Liesen2000}.

As a corollary, for given $\sigma > 1$, only finitely many Faber--Walsh polynomials $b_n(z)$ can have zeros on or outside $\Gamma_\sigma$ (cf. \cite{Walsh1958}).

We end this section by a proposition showing that the exterior mapping function can be ``recovered'' from the Faber--Walsh polynomials (and the sequence $(\alpha_j)_{j=1}^\infty$). This is a generalisation of the following property of Faber polynomials: The quotient of two succeeding Faber polynomials for the compact set $E$, $\tfrac{F_{n+1}(z)}{F_n(z)}$, converges locally uniformly in $K = \widehat{\C} \backslash E$ to the exterior mapping function, see for instance \cite{Curtiss1971}, \cite[pp.~134-135]{SmirnovLebedev} or \cite[p.~43]{Suetin1998}.

\begin{prop} \label{prop:fw_poly_quotient}
In the notation of theorems~\ref{thm:existence_ext_mapping} and \ref{thm:existence_fw_poly}, the Faber--Walsh polynomials $b_n(z)$ associated with $E$ and $(\alpha_j)_{j=1}^\infty$ satisfy
\begin{equation*}
\tfrac{ b_{n+1}(z) }{ b_n(z) } + \alpha_{n+1} \to \Phi(z), \quad n \to \infty,
\end{equation*}
locally uniformly in $\widehat{\C} \backslash E$.
\end{prop}

\begin{proof}
Let $1 < \lambda < \sigma < \infty$, $\tau$ with $\abs{U(\tau)} \geq \sigma \mu$ and $z = \psi(\tau)$.  Then, using \eqref{eqn:asymptotics_of_bn},
\begin{align}
\abs{ \tfrac{ b_{n+1}(z)}{b_n(z)} + \alpha_{n+1} - \Phi(z) }
&= \abslr{ \tfrac{ u_{n+1}(\tau) }{ u_n(\tau) } \tfrac{ 1 + \theta_{n+1}(z) G(\lambda) \tfrac{ (\lambda \mu)^{n+1} }{ \abs{u_{n+1}(\tau)} } }{ 1 + \theta_n(z) G(\lambda) \tfrac{ (\lambda \mu)^n }{ \abs{u_n(\tau)} } } - (\tau-\alpha_{n+1}) } \nonumber \\
&= \abs{\tau-\alpha_{n+1}} \abslr{ \tfrac{ \theta_{n+1}(z) G(\lambda) \frac{ (\lambda \mu)^{n+1} }{\abs{u_{n+1}(\tau)}} - \theta_n(z) G(\lambda) \frac{(\lambda \mu)^n}{ \abs{u_n(\tau)} }}{1 + \theta_n(z) G(\lambda) \frac{(\lambda \mu)^n}{ \abs{u_n(\tau)} } } } \nonumber \\
&\leq \abs{\tau-\alpha_{n+1}} \tfrac{ G(\lambda) \frac{ (\lambda \mu)^{n+1} }{\abs{u_{n+1}(\tau)}} + G(\lambda) \frac{(\lambda \mu)^n}{ \abs{u_n(\tau)} }}{1 - G(\lambda) \frac{(\lambda \mu)^n}{ \abs{u_n(\tau)} } }. \label{eqn:bound_bkp1/bk+alphak-Phi}
\end{align}
Here we used that $\theta_n(z)$ from \eqref{eqn:def_theta_n} satisfies $\abs{\theta_n(z)} \leq 1$, since $\abs{U(\tau)} \geq \lambda \mu$.
Now, by lemma~\ref{lem:existence_alpha_n} there exists $A_1 > 0$ such that $A_1 < \tfrac{ \abs{u_n(\tau)} }{ \abs{U(\tau)}^n }$ for all $\tau$ with $\abs{U(\tau)} \geq \sigma \mu$ and all $n$.  This estimate implies $\tfrac{1}{ \abs{u_n(\tau)} } < \tfrac{1}{A_1} \tfrac{1}{ \abs{U(\tau)}^n } \leq \tfrac{1}{A_1} \tfrac{1}{ (\sigma \mu)^n }$.
Therefore \eqref{eqn:bound_bkp1/bk+alphak-Phi} yields
\begin{equation*}
\abs{ \tfrac{ b_{n+1}(z)}{b_n(z)} + \alpha_{n+1} - \Phi(z) }
\leq 
\tfrac{ G(\lambda) }{ 1 - G(\lambda) \frac{1}{A_1} ( \frac{\lambda \mu}{\sigma \mu} )^n }
\abs{\tau-\alpha_{n+1}} \Big( \tfrac{ (\lambda \mu)^{n+1} }{\abs{u_{n+1}(\tau)}} + \tfrac{(\lambda \mu)^n}{ \abs{u_n(\tau)} } \Big).
\end{equation*}
Note that
\begin{equation*}
\abs{\tau-\alpha_{n+1}} \Big( \tfrac{ (\lambda \mu)^{n+1} }{\abs{u_{n+1}(\tau)}} + \tfrac{(\lambda \mu)^n}{ \abs{u_n(\tau)} } \Big) = \tfrac{ (\lambda \mu)^{n+1} }{ \abs{u_n(\tau)} } + \tfrac{ (\lambda \mu)^n }{ \abs{u_{n-1}(\tau)} } \tfrac{ \abs{\tau-\alpha_{n+1}} }{ \abs{\tau-\alpha_n} },
\end{equation*}
and that $\tfrac{ \abs{\tau-\alpha_{n+1}} }{ \abs{\tau-\alpha_n} }$ is globally bounded on $\{ \tau \in \widehat{\C} : \abs{U(\tau)} \geq \sigma \mu \}$, say by $C > 0$.  Thus
\begin{equation*}
\abs{ \tfrac{ b_{n+1}(z)}{b_n(z)} + \alpha_{n+1} - \Phi(z) }
\leq 
\tfrac{ G(\lambda) }{ 1 - G(\lambda) \frac{1}{A_1} ( \frac{\lambda \mu}{\sigma \mu} )^n } \tfrac{1+C}{A_1} \tfrac{ (\lambda \mu)^n }{ (\sigma \mu)^{n-1} }.
\end{equation*}
This shows that $\tfrac{ b_{n+1}(z) }{ b_n(z) } + \alpha_{n+1}$ converges uniformly to $\Phi(z)$ on $\psi( \{ \tau \in \widehat{\C} : \abs{U(\tau)} \geq \sigma \mu \} )$, for any $\sigma > 1$.  In particular, the convergence is locally uniformly for $z$ in $\widehat{\C} \backslash E$.
\end{proof}

\subsection{Recursion formula and computation of Faber--Walsh polynomials} \label{sect:recursion}

Walsh showed in \cite{Walsh1958} that the coefficients $b_n(z)$ in \eqref{eqn:defn_bk} are indeed (monic) polynomials of degree $n$.  With little extra effort, the proof yields also an algorithm for the computation of $b_0(z), b_1(z), \ldots$ from the sequence $(\alpha_j)_{j=1}^\infty$ and the coefficients of the Laurent series of $\psi$ at infinity.  For convenience and to set notation, we briefly recall Walsh's proof while deriving this algorithm.

The main ingredient is the analysis of the function $N_n(\tau,z)$ (notation as in \cite{Walsh1958}) defined by
\begin{equation}
\tfrac{ \psi'(\tau) }{ \psi(\tau) - z } - \tfrac{ b_0 (z) }{ u_1(\tau) } - \tfrac{ b_1 (z) }{ u_2(\tau) } - \ldots - \tfrac{ b_n(z) }{ u_{n+1}(\tau) } = \tfrac{ N_n (\tau,z) }{ (\psi(\tau)-z) (\tau-\alpha_1) (\tau-\alpha_2) \ldots (\tau-\alpha_{n+1} ) }, \label{eqn:definition_N_n_tau_z}
\end{equation}
where $\tau \in D_\sigma^\infty$ und $z \in \Gamma_\sigma$.  By \eqref{eqn:defn_bk}, $\tau = \infty$ is a zero of order at least $n+2$ of the right hand side of \eqref{eqn:definition_N_n_tau_z}.  Therefore, if $N_n(\tau,z)$ is expanded into a series with respect to $\tau$, all coefficients of $\tau^k$, $k \geq 1$, vanish.  Let
\begin{equation}
\psi(\tau) = \tau + c_0 + \sum_{k=1}^\infty \tfrac{c_k}{\tau^k} \label{eqn:coeff_of_psi}
\end{equation}
in a neighbourhood of infinity (recall \eqref{eqn:def_psi}).

We show the following lemma, which is essentially due to Walsh \cite{Walsh1958}.

\begin{lem}
With $N_n(\tau,z)$ from \eqref{eqn:definition_N_n_tau_z} and $\psi$ as in \eqref{eqn:coeff_of_psi}, we have for all $n \geq 0$
\begin{enumerate}
\item $b_n(z)$ is a monic polynomial in $z$ of degree $n$.

\item $N_n(\tau,z)$ is a monic polynomial in $z$ of degree $n+1$.  With respect to $\tau$, $N_n(\tau,z)$ satisfies
\begin{equation}
N_n(\tau,z) = \beta_0^{(n)}(z) \tau^0 + \sum\limits_{k=1}^\infty \tfrac{\beta_k^{(n)}(z)}{\tau^k}, \label{eqn:Series_N_n_tau_z}
\end{equation}
where $\beta_k^{(n)}(z)$, $k \geq 0$, are polynomials in $z$. $\beta_0^{(n)}(z)$ is monic of degree $n+1$, whereas the $\beta_k^{(n)}(z)$, $k \geq 1$, have degree at most $n$.

\item The polynomials $b_0(z)$, $b_1(z)$, $\ldots$, $b_n(z)$ are chosen in such a way that the coefficients of $\tau^k$, $k \geq 1$, in $N_n(\tau,z)$ vanish.
\end{enumerate}
\end{lem}

\begin{proof}
For $n=0$
\begin{equation*}
\tfrac{ \psi'(\tau) }{ \psi(\tau) - z} - \tfrac{ b_0 (z) }{ \tau-\alpha_1 } = \tfrac{ \psi'(\tau) (\tau-\alpha_1) - b_0(z) (\psi(\tau)-z) }{ (\psi(\tau)-z) (\tau-\alpha_1) } = \tfrac{ N_0 (\tau,z) }{ ( \psi(\tau) - z ) u_1(\tau) }.
\end{equation*}
With the Laurent series \eqref{eqn:coeff_of_psi} of $\psi$ at infinity, we obtain
\begin{align*}
N_0(\tau,z) &= \tau \psi'(\tau) - \alpha_1 \psi'(\tau) + b_0(z) z - b_0(z) \psi(\tau) \\
&= \tau ( 1 - b_0(z) ) + \tau^0 ( \underbrace{ b_0(z) z - b_0(z) c_0 - \alpha_1 }_{ \eqcolon \beta_0^{(0)}(z)} ) + \tfrac{1}{\tau} ( \underbrace{ -1 c_1 - c_1 b_0(z) }_{ \eqcolon \beta_1^{(0)}(z)} ) \\
&\phantom{=} + \sum\limits_{k=2}^\infty ( \underbrace{ \alpha_1 (k-1) c_{k-1} -k c_k - c_k b_0(z) }_{ \eqcolon \beta_k^{(0)}(z)} ) \tfrac{1}{\tau^k}.
\end{align*}
Since the coefficient of $\tau$ has to vanish, $b_0(z) = 1$.  Further $\beta_0^{(0)}(z) = z - c_0 - \alpha_1$ and $\beta_1^{(0)}(z) = - 2 c_1$ and, for $k\geq 2$, $\beta_k^{(0)}(z) = \alpha_1 (k-1) c_{k-1} - (k+1) c_k$. Thus $N_0(\tau,z)$ is as claimed.

Now suppose $b_0(z), b_1(z), \ldots, b_{n-1}(z)$ and $N_0(\tau,z), N_1(\tau,z), \ldots, N_{n-1}(\tau,z)$ have the desired properties.  Then, from equation \eqref{eqn:definition_N_n_tau_z} we have
\begin{equation*}
\tfrac{ N_n(\tau,z) }{ (\psi(\tau)-z) u_{n+1}(\tau) } = \tfrac{ N_{n-1} (\tau,z) }{ (\psi(\tau)-z) u_n(\tau) } - \tfrac{ b_n (z) }{ u_{n+1}(\tau) } = \tfrac{ N_{n-1}(\tau,z) (\tau-\alpha_{n+1}) - b_n(z) ( \psi(\tau)-z ) }{ (\psi(\tau)-z) u_{n+1}(\tau) },
\end{equation*}
hence
\begin{align*}
N_n(\tau,z)
&= \tau N_{n-1}(\tau,z) - \alpha_{n+1} N_{n-1}(\tau,z) + z b_n(z) - b_n(z) \psi(\tau) \\
&= \tau ( \beta_0^{(n-1)}(z) - b_n(z) ) + ( \underbrace{ \beta_1^{(n-1)}(z) - \alpha_{n+1} \beta_0^{(n-1)}(z) - b_n(z) c_0 + b_n(z) z }_{ \eqcolon \beta_0^{(n)}(z) } ) \\
&\phantom{=} + \sum\limits_{k=1}^\infty ( \underbrace{ \beta_{k+1}^{(n-1)}(z) - \alpha_{n+1} \beta_k^{(n-1)}(z) - b_n(z) c_k }_{ \eqcolon \beta_k^{(n)}(z) } ) \tfrac{1}{\tau^k}.
\end{align*}
Thus $b_n(z) = \beta_0^{(n-1)}(z)$, since the coefficient of $\tau$ has to vanish.  By the induction hypothesis on $\beta_0^{(n-1)}(z)$, $b_n(z)$ is a monic polynomial in $z$ of degree $n$.  Further,
\begin{equation*}
\beta_0^{(n)}(z) = \underbrace{ \beta_1^{(n-1)}(z) }_{ \deg(...) \leq n-1} - \underbrace{ \alpha_{n+1} \beta_0^{(n-1)}(z) }_{ \deg(...) \leq n } - \underbrace{ b_n(z) c_0 }_{ \deg(...) \leq n } + \underbrace{ b_n(z) z }_{ \deg(...) = n+1 }
\end{equation*}
is a monic polynomial in $z$ of degree $n+1$.  For given $k \geq 1$,  $\beta_k^{(n)}(z)$ is a polynomial in $z$ with $\deg(\beta_k^{(n)}(z)) \leq n$.  Thus $N_n(\tau,z)$ also is a monic polynomial in $z$ of degree $n+1$.
\end{proof}

As an immediate consequence we get a recursive formula for the computation of the Faber--Walsh polynomials.

\begin{prop} \label{prop:recursion_formula_bn}
Let $N_n(\tau,z)$ be defined by \eqref{eqn:definition_N_n_tau_z} with the representation \eqref{eqn:Series_N_n_tau_z}, i.e.
\begin{equation*}
N_n(\tau,z) = \beta_0^{(n)}(z) \tau^0 + \sum\limits_{k=1}^\infty \tfrac{\beta_k^{(n)}(z)}{\tau^k},
\end{equation*}
and let the coefficients $c_k$ of $\psi$ be given by \eqref{eqn:coeff_of_psi}.  Then the following formulae for the coefficients of $N_n(\tau,z)$ hold:
\begin{align}
\beta_0^{(0)}(z) &= (z-c_0) b_0(z) - \alpha_1, \label{eqn:beta_0^0} \\
\beta_k^{(0)}(z) &= - c_k b_0(z) - k c_k - (k-1) c_{k-1} (-\alpha_1), \quad k \geq 1, \label{eqn:beta_k^0}
\end{align}
and for $n \geq 1$
\begin{align}
\beta_0^{(n)}(z) &= (z-c_0) b_n(z) + \beta_1^{(n-1)}(z) + \beta_0^{(n-1)}(z) (- \alpha_{n+1}), \label{eqn:recursion_beta_0^n}\\
\beta_k^{(n)}(z) &= \phantom{(z}-c_k \: b_n(z) + \beta_{k+1}^{(n-1)}(z) + \beta_k^{(n-1)}(z) (-\alpha_{n+1}), \quad k \geq 1. \label{eqn:recursion_beta_k^n}
\end{align}
Further, $b_0(z) = 1$ and for $n \geq 0$ the Faber--Walsh polynomials are given by
\begin{equation}
b_{n+1}(z) = \beta_0^{(n)}(z). \label{eqn:Zshg_bn_und_beta_0^n}
\end{equation}
\end{prop}

These formulae show that the coefficients $\beta_k^{(n)}(z)$, $k \geq 0$, can be computed with $b_n(z)$.  Then, the next Faber--Walsh polynomial $b_{n+1}(z)$ is obtained from the equation $\beta_0^{(n)}(z) = b_{n+1}(z)$.  This allows a recursive computation of the Faber--Walsh polynomials.
The relations \eqref{eqn:beta_0^0}--\eqref{eqn:recursion_beta_k^n} can be visualised in a triangular scheme.  For clarity of presentation we indicate neither the dependence of $\beta_k^{(n)}(z)$ on $b_n(z)$, i.e. $\beta_0^{(n-1)}(z)$, nor on the sequence $(\alpha_j)_{j=1}^\infty$.  The scheme then looks as follows:
\begin{center}
\begin{pgfpicture}{0cm}{-4.5cm}{8cm}{0.5cm}
\pgfsetxvec{\pgfpoint{2.3cm}{0cm}}
\pgfnodebox{c_0}[virtual]{\pgfxy(0,0)}{$c_0$}{2pt}{2pt}
\pgfnodebox{c_1}[virtual]{\pgfxy(1,0)}{$c_1$}{2pt}{2pt}
\pgfnodebox{c_2}[virtual]{\pgfxy(2,0)}{$c_2$}{2pt}{2pt}
\pgfnodebox{c_3}[virtual]{\pgfxy(3,0)}{$c_3$}{2pt}{2pt}
\pgfnodebox{c_4}[virtual]{\pgfxy(4,0)}{$\ldots$}{2pt}{2pt}

\pgfnodebox{beta_0^0}[virtual]{\pgfxy(0,-1)}{$b_1(z) = \beta_0^{(0)}(z)$}{2pt}{2pt}
\pgfnodebox{beta_1^0}[virtual]{\pgfxy(1,-1)}{$\beta_1^{(0)}(z)$}{2pt}{2pt}
\pgfnodebox{beta_2^0}[virtual]{\pgfxy(2,-1)}{$\beta_2^{(0)}(z)$}{2pt}{2pt}
\pgfnodebox{beta_3^0}[virtual]{\pgfxy(3,-1)}{$\beta_3^{(0)}(z)$}{2pt}{2pt}
\pgfnodebox{beta_4^0}[virtual]{\pgfxy(4,-1)}{$\ldots$}{2pt}{2pt}

\pgfnodeconnline{c_0}{beta_0^0}
\pgfnodeconnline{c_0}{beta_1^0}
\pgfnodeconnline{c_1}{beta_1^0}
\pgfnodeconnline{c_1}{beta_2^0}
\pgfnodeconnline{c_2}{beta_2^0}
\pgfnodeconnline{c_2}{beta_3^0}
\pgfnodeconnline{c_3}{beta_3^0}

\pgfnodebox{beta_0^1}[virtual]{\pgfxy(1,-2)}{$b_2(z) = \beta_0^{(1)}(z)$}{2pt}{2pt}
\pgfnodebox{beta_1^1}[virtual]{\pgfxy(2,-2)}{$\beta_1^{(1)}(z)$}{2pt}{2pt}
\pgfnodebox{beta_2^1}[virtual]{\pgfxy(3,-2)}{$\beta_2^{(1)}(z)$}{2pt}{2pt}
\pgfnodebox{beta_3^1}[virtual]{\pgfxy(4,-2)}{$\ldots$}{2pt}{2pt}

\pgfnodeconnline{beta_0^0}{beta_0^1}
\pgfnodeconnline{beta_1^0}{beta_0^1}
\pgfnodeconnline{beta_1^0}{beta_1^1}
\pgfnodeconnline{beta_2^0}{beta_1^1}
\pgfnodeconnline{beta_2^0}{beta_2^1}
\pgfnodeconnline{beta_3^0}{beta_2^1}

\pgfnodebox{beta_0^2}[virtual]{\pgfxy(2,-3)}{$b_3(z) = \beta_0^{(2)}(z)$}{2pt}{2pt}
\pgfnodebox{beta_1^2}[virtual]{\pgfxy(3,-3)}{$\beta_1^{(2)}(z)$}{2pt}{2pt}
\pgfnodebox{beta_2^2}[virtual]{\pgfxy(4,-3)}{$\ldots$}{2pt}{2pt}

\pgfnodeconnline{beta_0^1}{beta_0^2}
\pgfnodeconnline{beta_1^1}{beta_0^2}
\pgfnodeconnline{beta_1^1}{beta_1^2}
\pgfnodeconnline{beta_2^1}{beta_1^2}

\pgfnodebox{beta_0^3}[virtual]{\pgfxy(3,-4)}{$b_4(z) = \beta_0^{(3)}(z)$}{2pt}{2pt}
\pgfnodebox{beta_1^3}[virtual]{\pgfxy(4,-4)}{$\ldots$}{2pt}{2pt}

\pgfnodeconnline{beta_0^2}{beta_0^3}
\pgfnodeconnline{beta_1^2}{beta_0^3}
\end{pgfpicture}
\end{center}
The Faber--Walsh polynomials are the diagonal entries of this scheme.  Note that the computation of $b_n(z) = \beta_0^{(n-1)}(z)$ requires only the values $\alpha_1, \ldots, \alpha_n$ and the coefficients $c_0, \ldots$, $c_{n-1}$ from the Laurent series of $\psi$ at infinity.

Note further that the recursion for the $\beta_k^{(n)}(z)$ has ``depth $1$'' in $n$. Hence, after computation of the relevant $\beta_k^{(n)}(z)$, all $\beta_k^{(n-1)}(z)$ can be eliminated (except $\beta_0^{(n-1)}(z) = b_n(z)$, of course).  However, if at some later point some more Faber--Walsh polynomials are to be computed, this is not recommendable, since for computing $b_{n+2}(z)$, all $\beta_k^{(j)}(z)$, $0 \leq k, j \leq n+1$, are needed. \bigskip

Proposition~\ref{prop:recursion_formula_bn} further allows to establish an explicit formula for $\beta_k^{(n)}(z)$ in terms of $b_0(z), b_1(z), \ldots, b_n(z)$, $c_0, \ldots, c_{n+k}$ and $\alpha_1, \ldots, \alpha_{n+1}$.  With \eqref{eqn:Zshg_bn_und_beta_0^n}, this also yields a formula for $b_{n+1}(z)$.  Unlike in the case of Faber polynomials, this formula turns out to be of little practical interest, due to its complexity.  Nevertheless, it may be of some theoretical interest.

\begin{lem} \label{lem:Formeln_beta_k^n}
The coefficients $\beta_k^{(n)}(z)$, $n \geq 0$, are given by
\begin{align*}
\beta_0^{(n)}(z) &= (z-c_0) b_n(z) - \sum\limits_{j=1}^n c_j b_{n-j}(z) - n c_n \\
&\phantom{=} - \sum\limits_{\substack{i+\ell=n \\ i = -1, 0, \ldots, n-1}} i c_i \sum\limits_{1 \leq i_1 < \ldots < i_\ell \leq n+1} \prod\limits_{m=1}^\ell (-\alpha_{i_m}) \\
&\phantom{=} - \sum\limits_{\substack{i+j+\ell=n \\ i \geq 1, j \geq 0, \ell \geq 1}} c_i b_j(z) \sum\limits_{j+2 \leq i_1 < \ldots < i_\ell \leq n+1} \prod\limits_{m=1}^\ell (-\alpha_{i_m}) \\
&\phantom{=} + (z-c_0) \sum\limits_{j=0}^{n-1} b_j(z) \prod\limits_{m=j+2}^{n+1} (-\alpha_m), \\
\beta_k^{(n)}(z) &= - \sum\limits_{j=0}^n c_{k+j} b_{n-j}(z) - (k+n) c_{k+n} \\
&\phantom{=} - \sum\limits_{\substack{i+\ell=n \\ i = -1, 0, \ldots, n-1}} (k+i) c_{k+i} \sum\limits_{1 \leq i_1 < \ldots < i_\ell \leq n+1} \prod\limits_{m=1}^\ell (-\alpha_{i_m}) \\
&\phantom{=} - \sum\limits_{\substack{i+j+\ell=n \\ i \geq 0, j \geq 0, \ell \geq 1}} c_{k+i} b_j(z) \sum\limits_{ j+2 \leq i_1 < \ldots < i_\ell \leq n+1} \prod\limits_{m=1}^\ell (-\alpha_{i_m}), \quad k \geq 1,
\end{align*}
and where $c_{-1} \coloneq 1$.
\end{lem}

We omit the easy but tedious proof (by induction) of the lemma.

This result can be applied to the case of Faber polynomials ($\nu = 1$).  Recall that in this case $\alpha_j = 0$ for all $j \in \N$.  Thus the formula for $\beta_0^{(n)}(z)$ yields the well-known recursion for Faber polynomials.

\begin{cor}
For the classical Faber polynomials $F_0(z) = 1$ and for $n \geq 0$ holds
\begin{equation*}
F_{n+1}(z) = (z-c_0) F_n(z) - \sum\limits_{j=1}^n c_j F_{n-j}(z) - n c_n.
\end{equation*}
\end{cor}

\section{Asymptotic convergence factor and asymptotically optimal polynomials}

In this section it is shown that Faber--Walsh polynomials normalised in some point are asymptotically optimal polynomials.

Denote the set of polynomials (with complex coefficients) of degree at most $n$ by $\cP_n$.  For given $z_0 \in \C$ let $\cP_n(z_0) \coloneq \{ p \in \cP_n : p(z_0) = 1 \}$ be the set of polynomials from $\cP_n$ normalised in $z_0$.

\begin{defn}[{cf. \cite{EiermannNiethammer1983}}] \label{def:asymptotic_conv_fac} \label{def:asymptotic_opt_poly}
Let $E \subseteq \C$ be compact and $z_0 \in \C$.  Then the \emph{asymptotic convergence factor for polynomials from $\cP_n(z_0)$ associated with $E$} is defined by
\begin{equation*}
R_{z_0}(E) \coloneq \limsup_{n \to \infty} \Big( \inf_{ p \in \cP_n(z_0) } \norm{p}_E \Big)^{\frac{1}{n}}.
\end{equation*}
Here, $\norm{f}_E = \sup_{z \in E} \abs{ f(z) }$ denotes the supremum norm on $E$. A sequence $(p_n(z))_{n=0}^\infty$ with $p_n(z) \in \cP_n(z_0)$ is called \emph{asymptotically optimal for $E$ with respect to $z_0$} if
\begin{equation*}
\lim\limits_{n \to \infty} \norm{p_n(z)}_E^{\frac{1}{n}} = R_{z_0}(E)
\end{equation*}
holds.
\end{defn}

Note that for any compact set $E$ and $z_0 \in \C$ the inequality $R_{z_0}(E) \leq 1$ holds.  If $z_0 \in E$, then $R_{z_0}(E) = 1$.

We now turn our attention to compact sets $E \subseteq \C$ whose complement $K$ is connected and \emph{regular} in the sense that $K$ possesses a Green's function $G$ with pole at infinity (cf. e.g. \cite[p.~65]{Walsh-IntApprox}).  If $E$ is compact, has no isolated points, and if $K$ is of finite connectivity, then $K$ is regular (cf. \cite[p.~65]{Walsh-IntApprox}).  The level lines of Green's function are denoted by
\begin{equation}
\Gamma_\sigma = \{ z \in K : G(z) = \log(\sigma) \}, \quad \sigma > 1. \label{eqn:def_Gamma_sigma}
\end{equation}
We will consistently use this notation throughout this section.  For such compact sets the asymptotic convergence factor can be characterised by Green's function.

\begin{prop} \label{prop:R(C)=1/sigma_0}
Let $E \subseteq \C$ be a compact set whose complement $K$ is connected and regular and let $z_0 \in \C \backslash E$.  Then
\begin{equation*}
R_{z_0}(E) = \tfrac{1}{\sigma_0},
\end{equation*}
where $\sigma_0 > 1$ is defined by $G(z_0) = \log (\sigma_0)$, where $G$ is Green's function with pole at infinity for $K$.
\end{prop}

This proposition is proved in \cite[Theorem~11]{EiermannNiethammerVarga1985} for simply connected $K$.  Later in this article, a note states that the proposition also holds for connected and regular $K$, referring to \cite[ch.~4]{Walsh-IntApprox}.  
For completeness, we give a sketch of the proof in the general case.

\begin{proof}
Since $E$ is compact and $z_0 \notin E$, we have
\begin{equation*}
m \coloneq \min_{z \in E} \abs{z-z_0} > 0, \quad M \coloneq \max_{z \in E} \abs{z-z_0} < \infty.
\end{equation*}
Let $p_n(z) \in \cP_n(z_0)$ and $q_{n-1}(z) \in \cP_{n-1}$ with $p_n(z) = 1 - (z-z_0) q_{n-1}(z)$.  Then
\begin{equation*}
\abs{ p_n(z) } = \abs{z-z_0} \abs{ \tfrac{1}{z-z_0} - q_{n-1}(z) } \leq M \abs{ \tfrac{1}{z-z_0} - q_{n-1}(z) } \leq \tfrac{M}{m} \abs{ p_n(z) }, \quad z \in E,
\end{equation*}
so that $\norm{ p_n(z) }_E \leq M \norm{ \tfrac{1}{z-z_0} - q_{n-1}(z) }_E \leq \tfrac{M}{m} \norm{ p_n(z) }_E$.  Now, since 
the map $\cP_{n-1} \to \cP_n(z_0)$, $q_{n-1}(z) \mapsto p_n(z) = 1 - (z-z_0) q_{n-1}(z)$, is a bijection, we obtain
\begin{equation*}
\inf_{ p_n \in \cP_n(z_0) } \norm{ p_n(z) }_E \leq M \inf_{ q_{n-1} \in \cP_{n-1} } \norm{ \tfrac{1}{z-z_0} - q_{n-1}(z) }_E \leq \tfrac{M}{m} \inf_{ p_n \in \cP_n(z_0) } \norm{ p_n(z) }_E.
\end{equation*}
Taking $n$-th root and the limit superior yields
\begin{equation}
R_{z_0}(E) = \limsup_{n \to \infty} \Big( \inf_{ p_n \in \cP_n(z_0) } \norm{ p_n(z) }_E \Big)^{\frac{1}{n}} = \limsup_{n \to \infty} \Big( \inf_{ q_{n-1} \in \cP_{n-1} } \norm{ \tfrac{1}{z-z_0} - q_{n-1}(z) }_E \Big)^{\frac{1}{n}}. \label{eqn:R(C)_charact_q_{n-1}}
\end{equation}

Note that $f(z) \coloneq \tfrac{1}{z-z_0}$ is analytic on $E$ and in the interior of $\Gamma_{\sigma_0}$ (recall \eqref{eqn:def_Gamma_sigma}), but is not analytic in the interior of any $\Gamma_\sigma$ with $\sigma > \sigma_0$.  Then, by
\cite[Theorem~4.1]{Saff2010}, which follows from the results in \cite[ch.~4]{Walsh-IntApprox}, we have
\begin{equation}
\limsup_{n \to \infty} \Big( \inf_{ q_n \in \cP_n } \norm{ \tfrac{1}{z-z_0} - q_n(z) }_E \Big)^{\frac{1}{n}} = \tfrac{1}{\sigma_0}. \label{eqn:limsup_q_n_1/sigma_0}
\end{equation}
From \eqref{eqn:limsup_q_n_1/sigma_0} it is not difficult to see that also the right hand side of \eqref{eqn:R(C)_charact_q_{n-1}} equals $\tfrac{1}{\sigma_0}$.
\end{proof}

With this characterisation of $R_{z_0}(E)$, we show that Faber--Walsh polynomials are asymptotically optimal in the sense of definition~\ref{def:asymptotic_opt_poly}, which was not known previously.

\begin{prop} \label{prop:FW_asympt_opt}
Let the notation be as in theorem~\ref{thm:existence_fw_poly} and assume $z_0 \in \C \backslash E$.  Then there exists a unique $\sigma_0 > 1$ with $z_0 \in \Gamma_{\sigma_0}$.  Further, $R_{z_0}(E) = \tfrac{1}{\sigma_0}$ and $R_{z_0}(\overline{E_\sigma}) = \tfrac{\sigma}{\sigma_0}$ for any $1 < \sigma \leq \sigma_0$, and the following limits exist:
\begin{align}
\lim\limits_{n \to \infty} \normlr{ \tfrac{b_n(.)}{b_n(z_0)} }_E^{\frac{1}{n}} &= \tfrac{1}{\sigma_0} = R_{z_0}(E), \label{eqn:limit_normalised_b_n_on_E} \\
\lim\limits_{n \to \infty} \normlr{ \tfrac{b_n(.)}{b_n(z_0)} }_{\Gamma_\sigma}^{\frac{1}{n}} &= \tfrac{\sigma}{\sigma_0}, \quad \text{for any } \sigma > 1. \label{eqn:limit_normalised_b_n_on_Gamma_sigma}
\end{align}
In particular, Faber--Walsh polynomials normalised in $z_0$ are asymptotically optimal for $E$ and any $\overline{E_\sigma}$, $1 < \sigma \leq \sigma_0$.
\end{prop}

\begin{proof}
Note first that $E$ is compact and its complement $K$ is connected and regular ($E$ and $K$ satisfy the sufficient condition mentioned before proposition~\ref{prop:R(C)=1/sigma_0}).  Since $z_0 \in K$, there exists a unique $\sigma_0 > 1$ with $z_0 \in \Gamma_{ \sigma_0 }$. (Namely $\abs{ U(\Phi(z_0)) } = \sigma_0 \mu$.)  Thus $R_{z_0}(E) = \tfrac{1}{\sigma_0}$ by proposition~\ref{prop:R(C)=1/sigma_0}.

If $G$ denotes Green's function with pole at infinity for $K$, then Green's function with pole at infinity for the complement of $\overline{E_\sigma}$ is $G_\sigma(z) = G(z) - \log(\sigma)$.  Hence, for $1 < \sigma \leq \sigma_0$, $G_\sigma(z_0) = \log( \tfrac{\sigma_0}{\sigma} )$ and $R(\overline{E_\sigma}) = \tfrac{\sigma}{\sigma_0}$ by proposition~\ref{prop:R(C)=1/sigma_0}.

Let $\sigma > 1$.  By proposition~\ref{prop:bn_two-sided_inequality} there exist $C_1, C_2 > 0$ and $n_\sigma \in \N$ such that
\begin{equation*}
C_1 \abs{ u_n(t) } < \abs{ b_n(\psi(t)) } < C_2 \abs{u_n(t)} \quad \forall\, t \in \Lambda_\sigma, n \geq n_\sigma.
\end{equation*}
Apply lemma~\ref{lem:existence_alpha_n} to bound $\abs{u_n(t)}$: There exist $A_1(\Lambda_\sigma), A_2(\Lambda_\sigma) > 0$ such that \eqref{eqn:double_bound_un} holds.  We have for $t \in \Lambda_\sigma$ (i.e. $\abs{U(t)} = \sigma \mu$), $z = \psi(t) \in \Gamma_\sigma$, and $n \geq n_\sigma$ the estimate
\begin{equation}
C_1 A_1(\Lambda_\sigma) (\sigma \mu)^n < \abs{ b_n(\psi(t)) } < C_2 A_2(\Lambda_\sigma) (\sigma \mu)^n. \label{eqn:estimate_of_b_n(z)}
\end{equation}
This shows $b_n(z) \neq 0$ for $z \in \Gamma_\sigma$ and $n \geq n_\sigma$.  In particular $b_n(z_0) \neq 0$ for $n \geq n_{\sigma_0}$.  From \eqref{eqn:estimate_of_b_n(z)} follows $\lim\limits_{n \to \infty} \abs{ b_n(z) }^{\frac{1}{n}} = \sigma \mu$ for any $z \in \Gamma_\sigma$, and in particular
\begin{equation}
\lim\limits_{n \to \infty} \abs{ b_n(z_0) }^{\frac{1}{n}} = \sigma_0 \mu. \label{eqn:limit_b_(0)^1/n}
\end{equation}
Further $\lim\limits_{n \to \infty} \norm{ b_n(.) }_{ \Gamma_\sigma }^{\frac{1}{n}} = \sigma \mu$, since \eqref{eqn:estimate_of_b_n(z)} holds uniformly for $z \in \Gamma_\sigma$.  This establishes \eqref{eqn:limit_normalised_b_n_on_Gamma_sigma}.

Now, let $\mu_n \coloneq \inf \{ \norm{p(.)}_E : p(z) \text{ monic of degree } n \}$.  From \cite[sect.~1.3.4]{SmirnovLebedev} it is known that $\big( \mu_n^{\frac{1}{n}} \big)_n$ converges to the capacity $\mu$ of $E$.  Since the Faber--Walsh polynomials are monic of degree $n$, we have the estimate
\begin{equation*}
\mu_n \leq \norm{b_n(.)}_E \leq \norm{b_n(.)}_{E_\sigma} = \norm{b_n(.)}_{\Gamma_\sigma} \leq A(\sigma) (\sigma \mu)^n
\end{equation*}
where $\sigma > 1$ is arbitrary and $A(\sigma) > 0$ is some constant (cf. \eqref{eqn:estimate_of_b_n(z)}).  This shows
\begin{equation*}
\mu \leq \liminf_{n \to \infty} \norm{b_n(.)}_E^{\frac{1}{n}} \leq \limsup_{n \to \infty} \norm{b_n(.)}_E^{\frac{1}{n}} \leq \sigma \mu
\end{equation*}
and $\lim\limits_{n \to \infty} \norm{ b_n(.) }_E^{\frac{1}{n}} = \mu$, since $\sigma > 1$ was arbitrary.  Combining this and \eqref{eqn:limit_b_(0)^1/n} shows \eqref{eqn:limit_normalised_b_n_on_E}.
\end{proof}

\begin{rem}
Another way to formulate proposition~\ref{prop:FW_asympt_opt} is that $\tfrac{ b_n(z) }{ b_n(z_0) } - 1$ is asymptotically an optimal solution to the approximation problem $1 \approx \sum_{k=1}^n a_k (z-z_0)^k$ on $E$.
Such approximation problems often occur in the analysis of iterative methods for solving linear algebraic systems. 
For instance, for semi-iterative methods asymptotically optimal polynomials normalised in $z_0 = 1$ play an important role, see e.g. \cite{EiermannNiethammer1983} (where this concept was actually introduced) and \cite{EiermannNiethammerVarga1985}.
Further, polynomials normalised in $z_0 = 0$ appear in the convergence analysis of the CG-, MINRES- and GMRES-methods, see e.g. \cite{Greenbaum1997}, \cite{LiesenStrakos2013} or \cite{Saad2003}.
\end{rem}

\section{An Example}

The major difficulty using Faber--Walsh polynomials is, of course, to find the exterior mapping function $\psi$.
To our knowledge, no example has been published in the literature so far.

An exterior mapping function can be explicitly constructed for the exterior of two (real) intervals symmetric with respect to zero (actually, it is sufficient that both intervals have same length, cf. corollary~\ref{cor:ext_map_for_linear_trafo_of_2_int}).

\subsection{Exterior mapping function}

The exterior mapping function for the exterior of two real intervals symmetric with respect to zero is introduced.

\begin{prop} \label{prop:ext_mapping_for_lemniscate}
Let $E = [-\beta, -\alpha] \cup [\alpha, \beta]$ with $0 < \alpha < \beta$.  Set $a \coloneq \tfrac{\beta+\alpha}{2}$ and $\mu \coloneq \sqrt{\tfrac{\beta-\alpha}{2} \tfrac{\beta+\alpha}{2}}$, so that $0 < \mu < a$ holds.
Then $z = \psi(w) = w \sqrt{1 + \frac{\mu^4}{a^2} \frac{1}{w^2-a^2}}$ (with $\sqrt{1} = +1$) defines a conformal bijection
\begin{equation*}
\psi : K_1 = \{ w \in \widehat{\C} : \abs{w-a}^{\frac{1}{2}} \abs{w+a}^{\frac{1}{2}} > \mu \} \; \to \; \widehat{\C} \backslash E
\end{equation*}
with $\psi(\infty) = \infty$ and $\psi'(\infty) = 1$.  In particular $E$ has capacity $\mu$.
\end{prop}

\begin{proof}
Let $U(w) = (w-a)^{\frac{1}{2}} (w+a)^{\frac{1}{2}}$.  The map $\psi$ is constructed for the right half-plane $\re(w) > 0$, then extended to $\{ w \in \widehat{\C} : \abs{U(w)} > \mu \}$ by the Schwarz reflection principle.  
First, $w_1 = \psi(w) = w^2 - a^2$ maps $\{ w \in \widehat{\C} : \re(w) > 0, \abs{U(w)} > \mu \}$ onto $\{ w : \abs{w} > \mu^2 \} \backslash ]-\infty, -a]$.  Then $w_2 = \psi_2(w_1) = \tfrac{1}{2} ( w_1 + \tfrac{\mu^4}{w_1} )$ maps this set to the complement of $]-\infty, \psi_2(a)] \cup [-\mu^2, \mu^2]$.  Now, $w_3 = \psi_3(w_2) = w_2 + \tfrac{1}{2} ( a^2 + \tfrac{\mu^4}{a^2} )$ takes this set to the complement of $]-\infty, 0] \cup [\tfrac{\alpha^2}{2}, \tfrac{\beta^2}{2}]$, which is mapped by $z = \psi_4(w_3) = \sqrt{2 w_3}$ (principal value of the square root) to $\{ z : \re(z) > 0 \} \backslash [\alpha, \beta]$.  Each of these maps is a conformal bijection on the image of the previous one, so that $z = \psi_+(w) = (\psi_4 \circ \psi_3 \circ \psi_2 \circ \psi_1) (w)$ is a conformal bijection from $\{w \in \widehat{\C} : \re(w) > 0, \abs{U(w)} > \mu \}$ to $\{ z \in \widehat{\C} : \re(z) > 0, z \notin [\alpha, \beta] \}$.

Note that this map extends to the imaginary axis (and preserves its orientation).  Thus $\psi_+$ can be extended by the Schwarz reflection principle to $\psi$ as stated in the theorem.  Further, $\psi$ can be extended analytically to $\{ w : \abs{U(w)} \geq \mu \}$, but not conformally.
\end{proof}

Proposition~\ref{prop:ext_mapping_for_lemniscate} gives a new proof for the well-known fact that $E = [-\beta, -\alpha] \cup [\alpha, \beta]$ has logarithmic capacity $\tfrac{\sqrt{\beta^2-\alpha^2}}{2}$, cf. \cite[Exercise, p.~178]{Saff2010}.  We have an immediate (slight) generalisation of proposition~\ref{prop:ext_mapping_for_lemniscate}.

\begin{cor} \label{cor:ext_map_for_linear_trafo_of_2_int}
Let the notation be as in proposition~\ref{prop:ext_mapping_for_lemniscate} and let $T$ be a M\"{o}bius transformation ($T(z) = Az+B$ with $A \neq 0$).  Then $T(E)$ is conformally equivalent to the lemniscatic domain
\begin{equation*}
T(K_1) = \{ w \in \widehat{\C} : \abs{ w - T(a) }^{\frac{1}{2}} \abs{ w - T(-a) }^{\frac{1}{2}} > \abs{A} \mu \},
\end{equation*}
via the conformal map $T \circ \psi \circ T^{-1}$, which fulfills $(T \circ \psi \circ T^{-1})(\infty) = \infty$ and $(T \circ \psi \circ T^{-1})'(\infty) = 1$.
\end{cor}

In the notation of proposition~\ref{prop:ext_mapping_for_lemniscate}, the inverse of $\psi$ is given by
\begin{equation}
w = \Phi(z) = \sqrt{ \tfrac{z^2}{2} + \tfrac{a^2}{2} - \tfrac{\mu^4}{2 a^2} + \sqrt{ \tfrac{1}{4} (z^2 - a^2 - \tfrac{\mu^4}{a^2} )^2 - a^2} }, \label{eqn:Phi_in_the_example}
\end{equation}
with suitably chosen branches of the square root.  Now, $R_{z_0}(E)$, $z_0 \in \C \backslash E$, can easily be determined.
Indeed, $z_0 \in \Gamma_{\sigma_0}$ if and only if $\Phi(z_0) \in \Lambda_{\sigma_0}$, i.e. if $\sigma_0 \mu = \abs{U(\Phi(z_0))}$.  Here $\Phi(z_0)$ can be determined from \eqref{eqn:Phi_in_the_example}.  Hence, by proposition~\ref{prop:R(C)=1/sigma_0},
\begin{equation*}
R_{z_0}(E) = \tfrac{1}{\sigma_0} = \tfrac{\mu}{ \abs{U(\Phi(z_0))} } = \tfrac{\mu}{ \abs{\Phi(z_0)^2 - a^2 }^{\frac{1}{2}} }.
\end{equation*}
For the special case $z_0 = 0$ we have $\Phi(0) = 0$, so that
\begin{equation}
R_0(E) = \tfrac{1}{\sigma_0} = \tfrac{\mu}{a} = \tfrac{ \sqrt{\beta-\alpha} }{ \sqrt{\beta+\alpha} }. \label{eqn:R0E}
\end{equation}
(Note that this expression appears in a bound for the relative residual norm in step $k$ of the MINRES-method, see e.g. \cite[equation (3.15)]{Greenbaum1997}.)

Here, it seems preferable to work directly with the mapping $\psi$.  In the remainder of this subsection we determine the Laurent series of $\psi$ at infinity.  Setting $\tau = \tfrac{1}{w}$ we have
\begin{equation*}
\psi(w) = \psi(\tfrac{1}{\tau}) = \tfrac{1}{\tau} \sqrt{ 1 + \tfrac{\mu^4}{a^4} \tfrac{ (a \tau)^2}{ 1 - ( a \tau )^2} }.
\end{equation*}
For the sake of brevity set
$t \coloneq (a \tau)^2$.  Then, it is sufficient to determine the power series of $\sqrt{ 1 + \tfrac{\mu^4}{a^4} \frac{t}{1-t} }$ at $t_0 = 0$.  Clearly,
\begin{equation}
1 + \tfrac{\mu^4}{a^4} \tfrac{t}{1-t} = 1 + \sum\limits_{n=1}^\infty \tfrac{\mu^4}{a^4} t^n \label{eqn:Taylorreihe_des_Radikanten}
\end{equation}
for $t$ with $\abs{t} < 1$.  Write $\psi(w) = \tfrac{1}{\tau} (f \circ g)(a^2 \tau^2)$ with $f(w) = \sqrt{w}$ and $g(t) = 1 + \tfrac{\mu^4}{a^4} \tfrac{t}{t+1}$.  Then, using the formula of Fa\`{a} di Bruno the series expansion of $\psi$ can be directly computed.

\begin{prop} \label{prop:explizite_Formel_fuer_Laurentreihe}
The Laurent series of $\psi$ at infinity is
\begin{equation*}
\psi(w) = w + \sum\limits_{n=1}^\infty \Bigg( \sum\limits_{ (k_1, \ldots, k_n) \in T_n } \tfrac{1}{ k_1! \ldots k_n! } \big( \tfrac{\mu^4}{a^4} \big)^{k_1 + \ldots + k_n} \prod\limits_{j=0}^{ k_1 + \ldots + k_n - 1} (\tfrac{1}{2} - j ) \Bigg) a^{2n} \frac{1}{w^{2n-1}},
\end{equation*}
where $T_n = \{ (k_1, \ldots, k_n) \in \N_0^n : 1 k_1 + 2 k_2 + \ldots + n k_n = n \}$.
\end{prop}

However, this explicit formula seems of little practical value.  We derive a simple recursion formula for the coefficients of the series expansion \eqref{eqn:coeff_of_psi} of $\psi$.

\begin{prop} \label{prop:rekursive_Formel_fuer_Laurentreihe}
The Laurent series of $\psi$ at infinity is
\begin{equation*}
\psi(w) = w + \sum\limits_{n=1}^\infty d_n a^{2n} \tfrac{1}{w^{2n-1}},
\end{equation*}
where the coefficents $d_n$, $n \in \N$, are (uniquely) determined by the following recursion formula:
\begin{equation*}
\begin{aligned}
d_0 &\coloneq 1, \\
d_n &\coloneq \tfrac{1}{2} \tfrac{\mu^4}{a^4} - \tfrac{1}{2} \sum_{k=1}^{n-1} d_k d_{n-k}, \quad n \geq 1.
\end{aligned}
\end{equation*}
\end{prop}

\begin{proof}
Set
\begin{equation}
\sqrt{1 + \tfrac{\mu^4}{a^4} \tfrac{t}{1-t}} = \sum\limits_{n=0}^\infty d_n t^n. \label{eqn:Ansatz_Laurentreihe_von_psi}
\end{equation}
Then $d_0 = \sqrt{1} = +1$ (principal value of the square root).  Since $w \mapsto \sqrt{1 + \tfrac{\mu^4}{a^4} \tfrac{a^2}{w^2 - a^2}}$ is analytic in a neighbourhood of $\infty$, the Taylor series \eqref{eqn:Ansatz_Laurentreihe_von_psi} is absolutely and uniformly convergent in a neighbourhood of $0$ ($t = \tfrac{1}{w}$).  Thus, squaring \eqref{eqn:Ansatz_Laurentreihe_von_psi} yields, together with \eqref{eqn:Taylorreihe_des_Radikanten},
\begin{equation*}
1 + \sum\limits_{n=1}^\infty \tfrac{\mu^4}{a^4} t^n = \sum\limits_{n=0}^\infty \Big( \sum\limits_{k=0}^n d_k d_{n-k} \Big) t^n.
\end{equation*}
Equating coefficients for $n \geq 1$ shows
\begin{equation*}
\tfrac{\mu^4}{a^4} = \sum\limits_{k=0}^n d_k d_{n-k} = d_0 d_n + \sum\limits_{k=1}^{n-1} d_k d_{n-k} + d_n d_0 = 2 d_n + \sum\limits_{k=1}^{n-1} d_k d_{n-k},
\end{equation*}
hence $d_n = \tfrac{1}{2} \tfrac{\mu^4}{a^4} - \tfrac{1}{2} \sum_{k=1}^{n-1} d_k d_{n-k}$. Then, with $\tau = \tfrac{1}{w}$,
\begin{equation*}
\psi(w)
= \tfrac{1}{\tau} \sqrt{ 1 + \tfrac{\mu^4}{a^4} \tfrac{ (a \tau)^2 }{ 1 - (a \tau)^2 } }
= \tfrac{1}{\tau} \Big( 1 + \sum\limits_{n=1}^\infty d_n a^{2n} \tau^{2n} \Big)
= w + \sum\limits_{n=1}^\infty d_n a^{2n} \tfrac{1}{w^{2n-1}},
\end{equation*}
as claimed.
\end{proof}

\subsection{Faber--Walsh polynomials} \label{sect:ex_fw_polys}

Let $E$ be as in proposition~\ref{prop:ext_mapping_for_lemniscate}.  For the construction of Faber--Walsh polynomials, we need to pick a suitable sequence $(\alpha_j)_{j=1}^\infty$ from the foci $a, -a$ of the lemniscatic domain (cf. section~\ref{sect:fw_polys}).  Here, one may choose for example one of the sequences $(a, -a, a, -a, \ldots)$ or $(-a, a, -a, a, \ldots)$.  (This follows easily from the construction of the sequence $(\alpha_j)_{j=1}^\infty$ as indicated in \cite{Walsh1958}.)  For both choices $u_{2k}(w) = (w-a)^k (w+a)^k$ and the associated Faber--Walsh polynomials $b_{2k}(z)$ are even polynomials.
This follows from the integral representation \eqref{eqn:integral_repres_bk} together with $u_{2k}(-w) = u_{2k}(w)$ and $\psi(-w) = -\psi(w)$.  Indeed, \eqref{eqn:integral_repres_bk} yields
\begin{equation*}
b_{2k}(-z) = \tfrac{1}{2 \pi i} \int_{\Lambda_\lambda} u_{2k}(\tau) \tfrac{ \psi'(\tau) }{ \psi(\tau) + z } \, d\tau
= \tfrac{1}{2 \pi i} \int_{\Lambda_\lambda} u_{2k}(-\tau) \tfrac{ \psi'(-\tau) }{ \psi(-\tau) - z } \, d(-\tau) = b_{2k}(z).
\end{equation*}
The last equality holds, since in this case, $t \mapsto \gamma(t)$ is a parametrisation of $\Lambda_\lambda$ if and only if $t \mapsto - \gamma(t)$ is a parametrisation of $\Lambda_\lambda$.  A similar argument for the Faber--Walsh polynomials with odd degree does not hold, see \eqref{eqn:first_5_fw_polys} below.

The next proposition shows that the Faber--Walsh polynomials do (in general) depend on the choice of the sequence $(\alpha_j)_{j=1}^\infty$.

\begin{prop}
Let $\alpha^+ = (a, -a, a, -a, \ldots)$, $\alpha^- = (-a, a, -a, a, \ldots)$.  Denote the polynomials $u_n(w)$ constructed from $\alpha^\pm$ by $u_n^\pm(w)$ and the associated Faber--Walsh polynomials by $b_n^\pm(z)$.  Then the following relations between $b_n^+(z)$ and $b_n^-(z)$ hold:
\begin{equation}
b_{2n}^+ (z) = b_{2n}^-(z), \quad b_{2n+1}^+ (-z) = - b_{2n+1}^-(z). \label{eqn:bnpm}
\end{equation}
In particular, for odd degree of $b_n(z)$, the coefficients for odd powers of $z$ are equal, the coefficients for even powers have opposite sign.
\end{prop}

\begin{proof}
The proposition follows from the integral representation \eqref{eqn:integral_repres_bk} of the Faber--Walsh polynomials, if one takes into account the relations $u_{2n}^+ (\tau) = u_{2n}^-(\tau)$, $u_{2n+1}^+ (\tau) = (\tau-a) u_{2n}^+(\tau)$ and $u_{2n+1}^-(\tau) = (\tau+a) u_{2n}^+(\tau)$.
\end{proof}

Now, let us pick the sequence $\alpha = \alpha^+ = (a, -a, a, -a, \ldots)$ for the rest of this section.

\begin{lem}
With the previous assumptions, the Faber--Walsh polynomials $b_n(z)$ associated with $E$ and $\alpha$ satisfy
\begin{align*}
b_{2k}(0) &= (-1)^k \tfrac{ a^{4k} + \mu^{4k} }{ a^{2k} } =  (-1)^k \big( a^{2k} + \big( \tfrac{\mu^2}{a} \big)^{2k} \big), \quad & k \geq 1, \\
b_{2k+1}(0) &= (-1)^{k-1} \tfrac{ a^{4k} + \mu^{4k} }{ a^{2k-1} } = (-1)^{k-1} a \big( a^{2k} + \big( \tfrac{\mu^2}{a} \big)^{2k} \big), \quad & k \geq 1.
\end{align*}
This shows $\abs{b_n(0)}^{\frac{1}{n}} \to a$ as $n \to \infty$, which also follows from \eqref{eqn:limit_b_(0)^1/n} and \eqref{eqn:R0E}.
\end{lem}

\begin{proof}[Sketch of the proof]
Note that $u_{2k}(\tau) = (\tau^2 - a^2)^k$ and $u_{2k+1}(\tau) = (\tau-a) (\tau^2 - a^2)^k$, by the choice of $\alpha$.  Further, $\psi$ from proposition~\ref{prop:ext_mapping_for_lemniscate} satisfies
\begin{equation*}
\tfrac{\psi'(\tau)}{\psi(\tau)} = \tfrac{1}{\tau} \tfrac{1}{ \tau^2 - a^2 + \frac{\mu^4}{a^2} } \tfrac{ (\tau^2 - a^2)^2 - \mu^4 }{ (\tau^2-a^2) }.
\end{equation*}
Now, the integral representation \eqref{eqn:integral_repres_bk} of the Faber--Walsh polynomials yields
\begin{equation*}
b_{2k}(0) = \tfrac{1}{2 \pi i} \int_{\Lambda_\lambda} u_{2k}(\tau) \tfrac{ \psi'(\tau) }{ \psi(\tau) } \, d\tau = \tfrac{1}{2 \pi i} \int_{ \Lambda_\lambda } \tfrac{ ( \tau^2 - a^2)^{k-1} \big( (\tau^2 - a^2)^2 - \mu^4 \big) }{ \tau \big( \tau^2 - a^2 + \frac{\mu^4}{a^2} \big) } \, d\tau.
\end{equation*}
The integrand is a meromorphic function with the three simple poles $\pm \sqrt{ \tfrac{a^4-\mu^4}{a^2} }$ and $0$, so that the last integral can be computed with the residue theorem.  In the same manner, also $b_{2k+1}(0)$ can be computed.
\end{proof}

Recall proposition~\ref{prop:rekursive_Formel_fuer_Laurentreihe} for the computation of the coefficients of the exterior mapping function $\psi$.  Using these coefficients and the recursion formula for the Faber--Walsh polynomials from section~\ref{sect:recursion}, we can compute (some of) the Faber--Walsh polynomials.  Let $E$ as above and $(\alpha_j)_{j=1}^\infty = (a, -a, a, -a, \ldots)$.  From proposition~\ref{prop:rekursive_Formel_fuer_Laurentreihe} we find $c_{2k} = 0$, $k \geq 0$, for the coefficients of the Laurent series at infinity of $\psi$.  Further,
\begin{equation*}
d_1 = \tfrac{1}{2} \tfrac{\mu^4}{a^4} - \tfrac{1}{2} \sum_{k=1}^0 d_k d_{1-k} = \tfrac{1}{2} \tfrac{\mu^4}{a^4}, \quad
d_2 = \tfrac{1}{2} \tfrac{\mu^4}{a^4} - \tfrac{1}{2} \sum_{k=1}^1 d_k d_{2-k} = \tfrac{1}{2} \tfrac{\mu^4}{a^4} - \tfrac{1}{2} \big( \tfrac{1}{2} \tfrac{\mu^4}{a^4} \big)^2.
\end{equation*}
Hence $c_1 = a^2 d_1 = \tfrac{1}{2} \tfrac{\mu^4}{a^2}$ and $c_3 = a^4 d_2 = \tfrac{1}{2} \mu^4 - \tfrac{1}{2} \big( \tfrac{1}{2} \tfrac{\mu^4}{a^2} \big)^2$.
Now, the recursion formula for the $b_n(z)$ from section~\ref{sect:recursion} yields, after lengthy computation,
\begin{equation}
\begin{aligned}
b_0(z) &= 1, \\
b_1(z) &= z-a, \\
b_2(z) &= z^2 - \tfrac{ a^4 + \mu^4 }{a^2}, \\
b_3(z) &= z^3 - a z^2 + (- a^2 - \tfrac{3}{2} \tfrac{\mu^4}{a^2} ) z + \tfrac{ a^4 + \mu^4}{a}, \\
b_4(z) &= z^4 - 2 \tfrac{a^4 + \mu^4}{a^2} z^2 + \tfrac{a^8 + \mu^8}{a^4}, \\
b_5(z) &= z^5 - a z^4 + ( - 2 a^2 - \tfrac{5}{2} \tfrac{\mu^4}{a^2} ) z^3 + ( 2 a^3 + 2 \tfrac{\mu^4}{a} ) z^2 + ( a^4 + \tfrac{1}{2} \mu^4 + \tfrac{15}{8} \tfrac{\mu^8}{a^4} ) z - \tfrac{a^8 + \mu^8}{a^3}.
\end{aligned} \label{eqn:first_5_fw_polys}
\end{equation}

\subsection{Approximation with truncated Faber--Walsh series}

This subsection contains numerical results on approximating $f(z) = \tfrac{1}{z}$ by a truncated Faber--Walsh series $s_n(z) = \sum_{k=0}^n a_k b_k(z)$, cf. theorem~\ref{thm:fw_series}, on $E = [-\beta, -\alpha] \cup [\alpha, \beta]$, $0 < \alpha < \beta$, as in proposition~\ref{prop:ext_mapping_for_lemniscate} .

Recall from proposition~\ref{prop:ext_mapping_for_lemniscate} that the complement $K$ of $E$ is conformally equivalent to the lemniscatic domain
\begin{equation*}
K_1 = \{ w \in \widehat{\C} : \abs{U(w)} = \abs{ w - a }^{\frac{1}{2}} \abs{ w + a }^{\frac{1}{2}} > \mu \}
\end{equation*}
by the conformal bijection $z = \psi(w) = w \sqrt{1 + \frac{\mu^4}{a^2} \frac{1}{w^2-a^2}}$.  Here $a = \tfrac{\beta+\alpha}{2}$ and $\mu = \sqrt{\tfrac{\beta-\alpha}{2} \tfrac{\beta+\alpha}{2}}$.  Let $(\alpha_j)_{j=1}^\infty = (a, -a, a, -a, \ldots)$ as in the previous subsection.

Consider the function $f(z) = \tfrac{1}{z}$ which is analytic on $E$.  Let us determine the number $\rho > 1$ from theorem~\ref{thm:fw_series}.  Since $f$ is analytic and single-valued in $\widehat{\C} \backslash \{ 0 \}$, the number $\rho$ is given by
\begin{equation*}
0 \in \Gamma_\rho \quad \Leftrightarrow \quad 0 = \psi^{-1}(0) \in \Lambda_\rho \quad \Leftrightarrow \quad \mu \rho = \abs{U(0)} = a \quad \Leftrightarrow \quad \rho = \tfrac{a}{\mu}.
\end{equation*}
Now, $f$ is analytic and single-valued at every point of $E_\rho$ (the interior of $\Gamma_\rho$), but in no $E_\sigma$ with $\sigma > \rho$, and $f$ has a singularity on $E_\rho$.

We are interested in the approximation of $f$ by a truncated Faber--Walsh series $s_n(z) = \sum_{k=0}^n a_k b_k(z)$, where $b_k(z)$ is the $k$-th Faber--Walsh polynomial associated with $E$ and $(\alpha_j)_{j=1}^\infty$, and the coefficients $a_k$ are given by \eqref{eqn:fw_series}, i.e.
\begin{equation*}
a_k = \tfrac{1}{2 \pi i} \int_{ \Lambda_\lambda } \tfrac{ f(\psi(\tau)) }{ u_{k+1}(\tau) } \, d\tau, \quad 1 < \lambda < \rho.
\end{equation*}

For a numerical example we let $E = [-\tfrac{5}{4}, - \tfrac{3}{4}] \cup [\tfrac{3}{4}, \tfrac{5}{4}]$, i.e. $a = 1$, $\mu = \tfrac{1}{2}$.  Then $\rho = \tfrac{a}{\mu} = 2$.  Note that $\Lambda_2$ is Bernoulli's lemniscate, and that $\Lambda_\lambda$ has two components for $1 < \lambda < 2$.
To determine $s_n(z)$, the coefficients $a_k$ from \eqref{eqn:fw_series} are needed.  In order to avoid a parametrisation of the generalised lemniscate $\Lambda_\lambda$, we replaced the component of $\Lambda_\lambda$ containing $\pm a$ by the polygon $P_\pm$ with the vertices $\pm a + r$, $\pm a + i r$, $\pm a - r$, $\pm a - i r$, for some $r>0$.  For both polygons, the point $\pm a + r$ was taken as starting and end point.  (The value of the integral remains unchanged if we deform the components of $\Lambda_\lambda$ to the polygon, as long as the deformation takes place in the domain of analyticity of the integrand.  In particular, $0$ must be an exterior point of $P_\pm$, while $\pm a$ must lie inside $P_\pm$.)
Suitable values of $r$ have been found to be $0.5 \leq r \leq 0.95$.  Then,
\begin{equation}
a_k = \tfrac{1}{2 \pi i} \int_{P_+} \tfrac{ f(\psi(\tau)) }{ u_{k+1}(\tau) } \, d\tau + \tfrac{1}{2 \pi i} \int_{P_-} \tfrac{ f(\psi(\tau)) }{ u_{k+1}(\tau) } \, d\tau. \label{eqn:ak_polygon_approx}
\end{equation}
All computations have been performed with MATLAB.  The Faber--Walsh polynomials $b_n(z)$ have been computed with the recursion formula from section~\ref{sect:recursion}, with the coefficients of $\psi$ from proposition~\ref{prop:rekursive_Formel_fuer_Laurentreihe}.  The integrals \eqref{eqn:ak_polygon_approx} were computed with MATLAB's \texttt{quadgk} command (adaptive Gauss-Kronrod quadrature), which supports integration over polygonal paths in the complex plane.
The coefficients $a_k$ of the Faber--Walsh series have been computed as \eqref{eqn:ak_polygon_approx} for $r = 0.7$.
The supremum norm of the error, $\norm{ f - s_n }_E$, has been evaluated by discretising both components of $E$ (meshwidth $0.01$), evaluating $\abs{ \tfrac{1}{z} - s_n(z) }$ for $z$ on the resulting grid and taking the maximum of these values.  Figure~\ref{fig:error_f-sn} shows the approximation error $\norm{f-s_n}_E$.

\begin{figure}[ht]
\includegraphics[width=\textwidth]{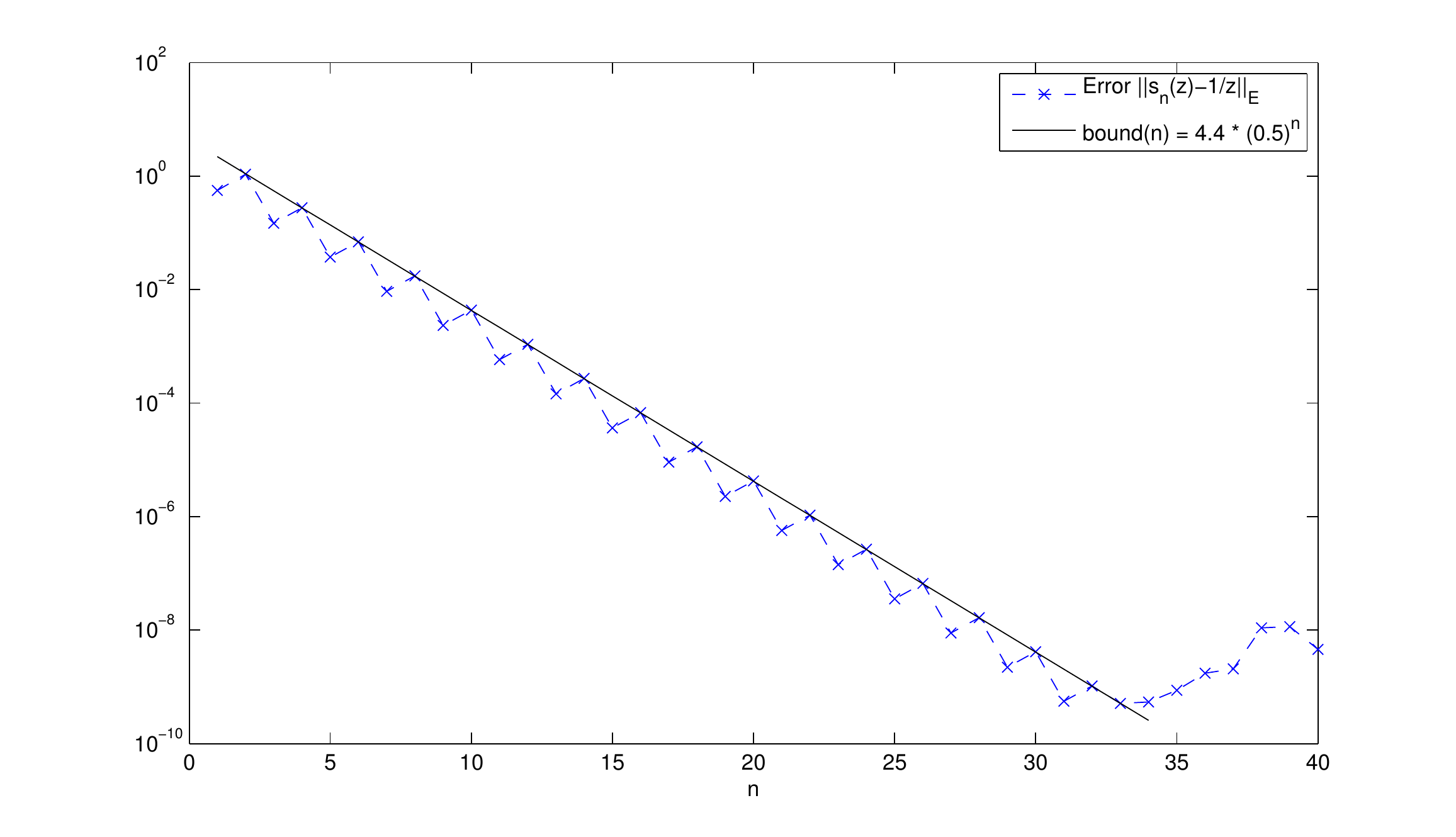}
\caption{Error $\norm{s_n(z)-\tfrac{1}{z}}_E$ for $n = 0, 1, \ldots, 40$ and $a = 1$, $\mu = \tfrac{1}{2}$ and $r = 0.7$, with the partial sums $s_n(z) = \sum_{k=0}^n a_k b_k(z)$ of the Faber--Walsh series \eqref{eqn:fw_series} for $f(z) = \tfrac{1}{z}$.}
\label{fig:error_f-sn}
\end{figure}

The minimal error $\norm{s_n(z) - \tfrac{1}{z}}_E$ is $5.1 \cdot 10^{-10}$ and occurs for degree $n = 33$.  For $n \geq 34$ the Faber--Walsh polynomials are strongly affected by round-off errors, which explains the stagnation (and even ``explosion'') of the error.  Note that for $n \leq 33$ the error is bounded by $4.4 \cdot \big( \tfrac{1}{2} \big)^n$, which nicely corresponds to the asymptotic behaviour predicted by \eqref{eqn:asymptotic_error_f-sn}.
Note also that the error decreases only every other step (more precisely for odd degrees $n$).  This is to be expected, since $f(z) = \tfrac{1}{z}$ is an odd function.
Thus, a polynomial with even degree is less likely to approximate $f$ on both intervals of $E$ at the same time, while a polynomial with odd degree should capture better the behaviour of $f$ on both intervals. \medskip

\textbf{Acknowledgements.} Thanks to Professor J\"{o}rg Liesen for helpful discussions and the careful reading of the manuscript.

\bibliographystyle{plain}
\bibliography{walsh}

\end{document}